\documentclass{article}%
\usepackage{amsmath}
\usepackage{amsfonts}
\usepackage{hyperref}
\usepackage{amssymb}
\usepackage{graphicx}
\usepackage{geometry}
\usepackage{amsthm}
\usepackage{appendix}
\usepackage{authblk}
\usepackage{cite}%
\providecommand{\U}[1]{\protect \rule{.1in}{.1in}}
\numberwithin{equation}{section}

\newtheorem{theorem}{Theorem}[section]

\newtheorem{corollary}[theorem]{Corollary}

\newtheorem{definition}[theorem]{Definition}

\newtheorem{lemma}[theorem]{Lemma}
\newtheorem{notation}[theorem]{Notation}

\newtheorem{proposition}[theorem]{Proposition}
\newtheorem{remark}[theorem]{Remark}

\newcommand{\arXiv}[1]{arXiv: \href{https://arXiv.org/abs/#1}{#1}}

\begin{document}

\title{\textbf{Explicit positive solutions to $G$-heat equations and the application
to $G$-capacities}}
\author{Mingshang Hu\thanks{Zhongtai Securities Institute for Financial Studies,
Shandong University. humingshang@sdu.edu.cn. Research supported by the
National Natural Science Foundation of China (No. 11671231) and the Young
Scholars Program of Shandong University (No. 2016WLJH10).}, \, \ Yifan
Sun\thanks{School of Mathematics, Shandong University.
sunyifan@mail.sdu.edu.cn.}}
\date{\today}
\maketitle

\begin{abstract}
In this paper, we study a class of explicit positive solutions to $G$-heat
equations by solving second order nonlinear ordinary differential equations.
Based on the positive solutions, we give the sharp order of $G$-capacity
$c_{\sigma}\left(  \left[  -\varepsilon,\varepsilon \right]  \right)  $ when
$\varepsilon \rightarrow0$.



\end{abstract}

\section{Introduction}

Peng
\cite{Peng(2004),Peng(2005),Peng(2007),Penglln(2007),Peng(2008),Pengclt(2008),Peng(2010)}
firstly proposed the theory of $G$-expectation which is a kind of dynamically
consistent sublinear expectation. In the theory of $G$-expectation, $G$-heat
equation plays an important role in the definitions of $G$-normal distribution
and $G$-Brownian motion. In this paper, we consider the following
1-dimensional $G$-heat equation:%

\begin{equation}
\partial_{t}u-G\left(  \partial_{xx}^{2}u\right)  =0,\, \, \, \left(
t,x\right)  \in \left[  0,+\infty \right)  \times%
\mathbb{R}
, \label{50}%
\end{equation}
with initial condition $u\left(  0,x\right)  =\phi \left(  x\right)  $. Here
$G\left(  a\right)  =\frac{1}{2}\left(  a^{+}-\sigma^{2}a^{-}\right)  $, where
$\sigma \in \left(  0,1\right]  $ is a fixed constant and $a^{+}=\max \left \{
0,a\right \}  $, $a^{-}=\left(  -a\right)  ^{+}$. Equation (\ref{50}) is also
called the Barenblatt equation (see
\cite{Barenblatt(1979),Barenblatt(1969),Kamin(1991)}). Particularly, when
$\sigma=1$, $G$-heat equation (\ref{50}) is the classical heat equation, which
corresponds to the normal distribution and Brownian motion in the linear
expectation theory.

In the sublinear expectation space, Peng developed the law of large numbers
and the central limit theorem (see \cite{Penglln(2007),Pengclt(2008)}). It is
worth pointing out that the limit distribution in central limit theorem is the
$G$-normal distribution, which is defined by $G$-heat equation. In
\cite{Peng(2007)}, Peng gave a formula to calculate the $G$-normal
distribution\ for convex or concave $\phi \in C_{b,lip}\left(  \mathbb{R}%
\right)  $, and the other special cases are dealt by Hu \cite{Hu(2012)} and
Song \cite{Song(2014)}.

The present\ article is devoted to find explicit positive solutions to the
$G$-heat equations with initial condition $u\left(  0,x\right)  =H\left(
x\right)  $ as follows:%
\[
u^{H}\left(  t,x\right)  =\left(  1+t\right)  ^{-\lambda}H\left(  \frac
{x}{\sqrt{1+t}}\right)  ,
\]
where $H\in C^{2}$ and $\lambda \in \left(  0,\frac{1}{2}\right)  $. For this
purpose, we obtain that $u^{H}$ is the solution to the $G$-heat equation if
and only if $H$ is the solution to the following second order nonlinear
ordinary differential equation (ODE for short):%

\begin{equation}
\left(  y^{\prime \prime}\right)  ^{+}-\sigma^{2}\left(  y^{\prime \prime
}\right)  ^{-}+xy^{\prime}+2\lambda y=0\text{.} \label{52}%
\end{equation}
Then, accordingly, $u^{H}$ can be obtained based on the positive solutions to
the above ODE.

We aim to get the positive solutions to ODE (\ref{52}) for each fixed
$\sigma \in \left(  0,1\right)  $. First, we fix a $\lambda \in \left(  0,\frac
{1}{2}\right)  $ and find a constant $\sigma_{\lambda}\in \left(  0,1\right)  $
depending on the given $\lambda$ such that ODE (\ref{52}) has positive
solutions for $\sigma \in \left[  \sigma_{\lambda},1\right)  $. Next we show
that $\sigma_{\lambda}$ is strictly increasing in $\lambda \in \left(
0,\frac{1}{2}\right)  $. Finally, based on this result, for each fixed
$\sigma \in \left(  0,1\right)  $, we obtain a constant $\lambda_{\sigma}%
\in \left(  0,\frac{1}{2}\right)  $ depending on the given $\sigma$ such that
ODE (\ref{52})\ has positive solutions for $\lambda \in \left(  0,\lambda
_{\sigma}\right]  $.

Through $G$-heat equation (\ref{50}), we can define $G$-capacity $c_{\sigma}$
(see Definition \ref{62104}). When $\sigma=1$, it is well-known that
\begin{equation}
\limsup_{\varepsilon \rightarrow0}\frac{c_{\sigma}\left(  \left[
-\varepsilon,\varepsilon \right]  \right)  }{\varepsilon}<+\infty \text{.}
\label{1601}%
\end{equation}
A natural question is: whether (\ref{1601}) holds for $\sigma \in \left(
0,1\right)  $. According to the explicit positive solutions to $G$-heat
equation (\ref{50}), for each fixed $\sigma \in \left(  0,1\right)  $, we obtain
that
\[
\limsup_{\varepsilon \rightarrow0}\frac{c_{\sigma}\left(  \left[
-\varepsilon,\varepsilon \right]  \right)  }{\varepsilon^{2\lambda_{\sigma}}%
}<+\infty
\]
and%
\begin{equation}
\liminf_{\varepsilon \rightarrow0}\frac{c_{\sigma}\left(  \left[
-\varepsilon,\varepsilon \right]  \right)  }{\varepsilon^{2\lambda}}%
=+\infty \text{ \ for }\lambda \in \left(  \lambda_{\sigma},\frac{1}{2}\right)
\text{,} \label{1602}%
\end{equation}
which implies that 2$\lambda_{\sigma}$ is the sharp order of $c_{\sigma
}\left(  \left[  -\varepsilon,\varepsilon \right]  \right)  $ when
$\varepsilon \rightarrow0$. We also prove that $\lambda_{\sigma}>\sigma^{2}/2$,
which improves the result of Example 3.9 in \cite{Hu(2016)}. Moreover, by
(\ref{1602}), we obtain that there is no positive solutions to $G$-heat
equation (\ref{50}) and ODE (\ref{52}) for $\lambda>\lambda_{\sigma}$.

This paper is organized as follows. Section \ref{s2} presents the notions and
properties of $G$-normal distribution and $G$-capacity. In Section \ref{s3},
the main results of the explicit positive solutions to ODEs and $G$-heat
equations are stated. Then we apply the results of Section \ref{s3} to give
the sharp order of $c_{\sigma}\left(  \left[  -\varepsilon,\varepsilon \right]
\right)  $ when $\varepsilon \rightarrow0$ in Section \ref{s4}. Section
\ref{s5} and Appendix provide some technical proofs of some lemmas in the
former sections.

\section{Basic Settings\label{s2}}

We present some basic notions and results of $G$-heat equation, $G$-normal
distribution and $G$-capacity. The readers may refer to
\cite{Denis(2011),Peng(2007),Peng(2010)} for more details.

We first define the $G$-normal distribution on $C_{b,lip}\left(
\mathbb{R}
\right)  $, where $C_{b,lip}\left(
\mathbb{R}
\right)  $ denotes the space of bounded and Lipschitz functions on
$\mathbb{R}$.

\begin{definition}
\label{62105}For $\sigma \in \left(  0,1\right]  $ and $t\geq0$, define
$\mathbb{\hat{E}}_{\sigma}^{t}:C_{b,lip}\left(  \mathbb{R}\right)
\rightarrow \mathbb{R}$ by
\begin{equation}
\mathbb{\hat{E}}_{\sigma}^{t}\left[  \phi \right]  :=u^{\phi}\left(
t,0\right)  , \label{62101}%
\end{equation}
where $u^{\phi}$ is the unique viscosity solution to the $G$-heat equation
(\ref{50}) with initial condition $u\left(  0,x\right)  =\phi \left(  x\right)
$. $\mathbb{\hat{E}}_{\sigma}^{t}\left[  \cdot \right]  $ is called the
$G$-normal distribution with variance uncertainty $\left[  \sigma
^{2}t,t\right]  $, denoted by $N\left(  0,\left[  \sigma^{2}t,t\right]
\right)  $. Particularly, when $t=1$, we write $\mathbb{\hat{E}}_{\sigma
}\left[  \cdot \right]  :=\mathbb{\hat{E}}_{\sigma}^{1}\left[  \cdot \right]  $
for simplicity of notation.
\end{definition}

\begin{remark}
When $\sigma=1$,%
\[
\mathbb{\hat{E}}_{1}^{t}\left[  \phi \right]  =u^{\phi}\left(  t,0\right)
=\frac{1}{\sqrt{2\pi t}}\int_{-\infty}^{+\infty}\phi \left(  x\right)
\exp \left \{  -\frac{x^{2}}{2t}\right \}  dx,
\]
which is the classical normal distribution $N\left(  0,t\right)  $.
\end{remark}

\begin{remark}
\label{2101}For $\sigma \in \left(  0,1\right]  $, $G$-heat equation (\ref{50})
is a uniformly parabolic PDE and $G$ is a convex function; it has a unique
$C^{1,2}$ solution (see \cite{Krylov(1987),Wang(1992)}).
\end{remark}

By the definition and the comparison theorem of the solutions to $G$-heat
equations, one can check that $\mathbb{\hat{E}}_{\sigma}^{t}\left[
\cdot \right]  $ satisfies the following properties: for each $\phi,\psi \in
C_{b,lip}\left(  \mathbb{R}\right)  $,

\begin{description}
\item[(a)] Monotonicity: $\mathbb{\hat{E}}_{\sigma}^{t}\left[  \phi \right]
\geq \mathbb{\hat{E}}_{\sigma}^{t}\left[  \psi \right]  $ if $\phi \geq \psi$;

\item[(b)] Cash translatability: $\mathbb{\hat{E}}_{\sigma}^{t}\left[
\phi+c\right]  =\mathbb{\hat{E}}_{\sigma}^{t}\left[  \phi \right]  +c$ for
$c\in%
\mathbb{R}
$;

\item[(c)] Sub-additivity: $\mathbb{\hat{E}}_{\sigma}^{t}\left[  \phi
+\psi \right]  \leq \mathbb{\hat{E}}_{\sigma}^{t}\left[  \phi \right]
+\mathbb{\hat{E}}_{\sigma}^{t}\left[  \psi \right]  $;

\item[(d)] Positive homogeneity: $\mathbb{\hat{E}}_{\sigma}^{t}\left[
\lambda \phi \right]  =\lambda \mathbb{\hat{E}}_{\sigma}^{t}\left[  \phi \right]
$ for $\lambda \geq0$.
\end{description}

We now give the relation between $\mathbb{\hat{E}}_{\sigma}^{t}\left[
\cdot \right]  $ and $\mathbb{\hat{E}}_{\sigma}\left[  \cdot \right]  $. Let
$u^{\phi}$ be the solution to $G$-heat equation (\ref{50}) with initial
condition $u\left(  0,x\right)  =\phi \left(  x\right)  $. For each fixed
$t\in \left(  0,+\infty \right)  $, define $\tilde{u}\left(  s,x\right)
=u^{\phi}\left(  ts,\sqrt{t}x\right)  $ for $\left(  s,x\right)  \in \left[
0,+\infty \right)  \times \mathbb{R}$. It is easy to verify that $\tilde{u}$ is
the solution to $G$-heat equation (\ref{50}) with initial condition $u\left(
0,x\right)  =\phi \left(  \sqrt{t}x\right)  $. By Definition \ref{62105}, we have%

\begin{equation}
\mathbb{\hat{E}}_{\sigma}^{t}\left[  \phi \left(  \cdot \right)  \right]
=\mathbb{\hat{E}}_{\sigma}\left[  \phi \left(  \sqrt{t}\cdot \right)  \right]  .
\label{62102}%
\end{equation}

Based on $\mathbb{\hat{E}}_{\sigma}^{t}\left[  \cdot \right]  $, we give the
definition of $G$-capacity.

\begin{definition}
\label{62104}For $\sigma \in \left(  0,1\right]  $, $t\geq0$ and $\left[
a,b\right]  \subseteq \mathbb{R}$, we define
\[
c_{\sigma}^{t}\left(  \left[  a,b\right]  \right)  :=\inf \left \{
\mathbb{\hat{E}}_{\sigma}^{t}\left[  \phi \right]  :\phi \geq I_{\left[
a,b\right]  },\phi \in C_{b,lip}\left(  \mathbb{R}\right)  \right \}  .
\]
Particularly, when $t=1$, we write $c_{\sigma}\left(  \left[  a,b\right]
\right)  :=c_{\sigma}^{1}\left(  \left[  a,b\right]  \right)  $ for simplicity
of notation.
\end{definition}

\begin{remark}
By (\ref{62102}), one can verify that%
\begin{equation}
c_{\sigma}^{t}\left(  \left[  a,b\right]  \right)  =c_{\sigma}\left(  \left[
\frac{a}{\sqrt{t}},\frac{b}{\sqrt{t}}\right]  \right)  \text{ \ for }t>0.
\label{62103}%
\end{equation}

\end{remark}

\begin{remark}
When $\sigma=1$,
\[
c_{1}^{t}\left(  \left[  a,b\right]  \right)  =\frac{1}{\sqrt{2\pi t}}\int
_{a}^{b}\exp \left \{  -\frac{x^{2}}{2t}\right \}  dx,
\]
which is the probability of normal distribution $N\left(  0,t\right)  $ on
$\left[  a,b\right]  $.
\end{remark}

\begin{remark}
\label{62108}If $\psi_{1}\leq I_{\left[  a,b\right]  }\leq \psi_{2}$ with
$\psi_{1},\psi_{2}\in C_{b,lip}\left(  \mathbb{R}\right)  $, then
$\mathbb{\hat{E}}_{\sigma}^{t}\left[  \psi_{1}\right]  \leq c_{\sigma}%
^{t}\left(  \left[  a,b\right]  \right)  \leq \mathbb{\hat{E}}_{\sigma}%
^{t}\left[  \psi_{2}\right]  $.
\end{remark}

\section{Explicit Positive Solutions to $G$-heat Equations\label{s3}}

In this section, our purpose is to find positive function $H\in C^{2}$ and
$\lambda \in \left(  0,\frac{1}{2}\right)  $ such that $u^{H}\left(  t,x\right)
=\left(  1+t\right)  ^{-\lambda}H\left(  x/\sqrt{1+t}\right)  $ is the
solution to $G$-heat equation (\ref{50}) with initial condition $u\left(
0,x\right)  =H\left(  x\right)  $ for any fixed $\sigma \in \left(  0,1\right)
$.

The following theorem gives the relation between the solutions to $G$-heat
equation (\ref{50}) and the solutions to ODE (\ref{52}).

\begin{theorem}
\label{2102}Let $\sigma \in \left(  0,1\right)  $ and $\lambda \in \left(
0,\frac{1}{2}\right)  $ be any fixed constants. Then $H$ is the $C^{2}$
solution to ODE (\ref{52}) if and only if%
\begin{equation}
u^{H}\left(  t,x\right)  =\left(  1+t\right)  ^{-\lambda}H\left(  \frac
{x}{\sqrt{1+t}}\right)  \label{54}%
\end{equation}
is the solution to G-heat equation (\ref{50}) with initial condition $u\left(
0,x\right)  =H\left(  x\right)  $.
\end{theorem}

\begin{proof}
We can easily check that%
\[
\, \partial_{t}u^{H}=\left(  1+t\right)  ^{-\lambda-1}\left[  -\lambda
H\left(  \frac{x}{\sqrt{1+t}}\right)  -\frac{1}{2}\frac{x}{\sqrt{1+t}%
}H^{\prime}\left(  \frac{x}{\sqrt{1+t}}\right)  \right]
\]
and%
\[
\partial_{xx}^{2}u^{H}=\left(  1+t\right)  ^{-\lambda-1}H^{\prime \prime
}\left(  \frac{x}{\sqrt{1+t}}\right)  \text{.}%
\]
Then we have%
\begin{equation}%
\begin{array}
[c]{cl}
& \partial_{t}u^{H}-G\left(  \partial_{xx}^{2}u^{H}\right) \\
= & -\frac{1}{2}\left(  1+t\right)  ^{-\lambda-1} \left[  2G\left(
H^{\prime \prime}\left(  \frac{x}{\sqrt{1+t}}\right)  \right)  +\frac{x}%
{\sqrt{1+t}}H^{\prime}\left(  \frac{x}{\sqrt{1+t}}\right)  +2\lambda H\left(
\frac{x}{\sqrt{1+t}}\right)  \right]  \text{.}%
\end{array}
\label{22040}%
\end{equation}

If $H$ is the $C^{2}$ solution to ODE (\ref{52}), then by (\ref{22040}) we
have $\partial_{t}u^{H}-G\left(  \partial_{xx}^{2}u^{H}\right)  =0$, so
$u^{H}\left(  t,x\right)  $ defined by (\ref{54}) is the solution to $G$-heat
equation (\ref{50}) with initial condition $u\left(
0,x\right)  =H\left(  x\right)  $.

If $u^{H}\left(  t,x\right)  $ in (\ref{54}) is the
solution to $G$-heat equation (\ref{50}) with initial condition $u\left(
0,x\right)  =H\left(  x\right)  $, then, according to (\ref{22040}), we have
\[
2G\left(  H^{\prime \prime}\left(  \frac{x}{\sqrt{1+t}}\right)  \right)
+\frac{x}{\sqrt{1+t}}H^{\prime}\left(  \frac{x}{\sqrt{1+t}}\right)  +2\lambda
H\left(  \frac{x}{\sqrt{1+t}}\right)  =0\text{.}%
\]
Because of the arbitrariness of $x\in%
\mathbb{R}
$ and $t\in \left[  0,+\infty \right)  $, $H$ is the solution to ODE (\ref{52}),
and we have $H\in C^{2}$ according to Remark \ref{2101}.
\end{proof}

For $\sigma=1$ and each fixed $\lambda \in \left(  0,\frac{1}{2}\right)  $, ODE
(\ref{52}) is a linear ODE%
\begin{equation}
y^{\prime \prime}+xy^{\prime}+2\lambda y=0. \label{55}%
\end{equation}
One can check that the general solution to ODE (\ref{55}) is
\begin{equation}
\Psi_{\lambda}\left(  x\right)  =\mu_{1}\varphi_{\lambda}\left(  x\right)
+\mu_{2}\varphi_{\lambda}\left(  -x\right)  , \label{2301}%
\end{equation}
where $\mu_{1}$ and $\mu_{2}$ are arbitrary constants, and%
\begin{equation}
\varphi_{\lambda}\left(  x\right)  :=\int_{0}^{+\infty}y^{-2\lambda}%
\exp \left \{  -\frac{\left(  y-x\right)  ^{2}}{2}\right \}  dy. \label{5333}%
\end{equation}
It is obvious that $\varphi_{\lambda}$ is positive on $%
\mathbb{R}
$, and it is easy to check that%
\begin{equation}
\varphi_{\lambda}^{\prime}\left(  x\right)  =\int_{0}^{+\infty}y^{-2\lambda
}\left(  y-x\right)  \exp \left \{  -\frac{\left(  y-x\right)  ^{2}}{2}\right \}
dy, \label{2401}%
\end{equation}
and%
\begin{equation}
\varphi_{\lambda}^{\prime \prime}\left(  x\right)  =\int_{0}^{+\infty
}y^{-2\lambda}\left[  \left(  y-x\right)  ^{2}-1\right]  \exp \left \{
-\frac{\left(  y-x\right)  ^{2}}{2}\right \}  dy\text{.} \label{2402}%
\end{equation}
Since $\varphi_{\lambda}$ satisfies ODE (\ref{55}), substituting (\ref{5333})
and (\ref{2401}) into ODE (\ref{55}), we also have
\begin{equation}
\varphi_{\lambda}^{\prime \prime}\left(  x\right)  =\int_{0}^{+\infty
}y^{-2\lambda}\left[  -2\lambda-x\left(  y-x\right)  \right]  \exp \left \{
-\frac{\left(  y-x\right)  ^{2}}{2}\right \}  dy\text{.} \label{3}%
\end{equation}

\begin{notation}
\label{3.2}For each fixed $\lambda \in \left(  0,\frac{1}{2}\right)  $, let
$x_{1}^{\lambda}<0$, $x_{2}^{\lambda}>0$ and $z^{\lambda}>0$ be the constants
such that
\[
\varphi_{\lambda}^{\prime \prime}\left(  x_{1}^{\lambda}\right)  =0\text{,
}\varphi_{\lambda}^{\prime \prime}\left(  x_{2}^{\lambda}\right)  =0\text{, and
}\varphi_{\lambda}^{\prime \prime}\left(  z^{\lambda}\right)  +\varphi
_{\lambda}^{\prime \prime}\left(  -z^{\lambda}\right)  =0\text{.}%
\]

\end{notation}

\begin{remark}
\label{2103}The existence and uniqueness of $x_{1}^{\lambda}$, $x_{2}%
^{\lambda}$ and $z^{\lambda}$ will be proved\ in Lemma \ref{2106} and Lemma
\ref{2107}. Moreover, we will show that $x_{2}^{\lambda}>z^{\lambda}%
>-x_{1}^{\lambda}$ in Lemma \ref{2107}.
\end{remark}

\begin{notation}
\label{3.4}Let%
\[
\sigma_{\lambda}:=-\frac{x_{1}^{\lambda}}{z^{\lambda}}\text{,}%
\]
where $x_{1}^{\lambda}$ and $z^{\lambda}$ are defined by Notation \ref{3.2}.
\end{notation}

Now we give the first main result of\ explicit positive solutions to ODE
(\ref{52}).

\begin{theorem}
\label{2104}Let $\lambda \in \left(  0,\frac{1}{2}\right)  $ be any fixed
constant, and let $\varphi_{\lambda}$ be defined by (\ref{5333}). Then for
each $\sigma \in \left[  \sigma_{\lambda},1\right)  $,%
\begin{equation}
H_{\lambda,\sigma}\left(  x\right)  :=\left \{
\begin{array}
[c]{lll}%
\mu_{1}^{\lambda}\varphi_{\lambda}\left(  x\right)  +\mu_{2}^{\lambda}%
\varphi_{\lambda}\left(  -x\right)  , &  & x\leq-\sigma z^{\lambda}\\
\varphi_{\lambda}\left(  \frac{x}{\sigma}\right)  +\varphi_{\lambda}\left(
-\frac{x}{\sigma}\right)  , &  & -\sigma z^{\lambda}<x<\sigma z^{\lambda}\\
\mu_{2}^{\lambda}\varphi_{\lambda}\left(  x\right)  +\mu_{1}^{\lambda}%
\varphi_{\lambda}\left(  -x\right)  , &  & x\geq \sigma z^{\lambda}%
\end{array}
\right.  \label{58}%
\end{equation}
is the positive $C^{2}$ solution to ODE (\ref{52}), where $\sigma_{\lambda}$
and $z^{\lambda}$ are defined by Notation \ref{3.4} and Notation \ref{3.2},
$\mu_{1}^{\lambda}>0$ and $\mu_{2}^{\lambda}\geq0$ are constants satisfying%
\begin{equation}
\left \{
\begin{array}
[c]{rl}%
\mu_{1}^{\lambda}= & \varphi_{\lambda}^{\prime \prime}\left(  \sigma
z^{\lambda}\right)  \cdot \left[  \varphi_{\lambda}\left(  -\sigma z^{\lambda
}\right)  \varphi_{\lambda}^{\prime \prime}\left(  \sigma z^{\lambda}\right)
-\varphi_{\lambda}\left(  \sigma z^{\lambda}\right)  \varphi_{\lambda}%
^{\prime \prime}\left(  -\sigma z^{\lambda}\right)  \right]  ^{-1}\cdot \left[
\varphi_{\lambda}\left(  -z^{\lambda}\right)  +\varphi_{\lambda}\left(
z^{\lambda}\right)  \right]  \text{,}\\
\mu_{2}^{\lambda}= & \varphi_{\lambda}^{\prime \prime}\left(  -\sigma
z^{\lambda}\right)  \cdot \left[  \varphi_{\lambda}\left(  \sigma z^{\lambda
}\right)  \varphi_{\lambda}^{\prime \prime}\left(  -\sigma z^{\lambda}\right)
-\varphi_{\lambda}\left(  -\sigma z^{\lambda}\right)  \varphi_{\lambda
}^{\prime \prime}\left(  \sigma z^{\lambda}\right)  \right]  ^{-1}\cdot \left[
\varphi_{\lambda}\left(  -z^{\lambda}\right)  +\varphi_{\lambda}\left(
z^{\lambda}\right)  \right]  \text{.}%
\end{array}
\right.  \label{59}%
\end{equation}

\end{theorem}

In order to prove Theorem \ref{2104}, we need the following lemmas.

\begin{lemma}
\label{2105}Let $\lambda \in \left(  0,\frac{1}{2}\right)  $ be any fixed
constant. If $\phi \left(  \cdot \right)  $ is a $C^{2}$ solution to ODE
(\ref{55}), then for each $\sigma \in \left(  0,1\right)  $, $\Phi \left(
x\right)  :=\phi \left(  x/\sigma \right)  $ is a $C^{2}$ solution to the
following ODE:%
\begin{equation}
\sigma^{2}y^{\prime \prime}+xy^{\prime}+2\lambda y=0\text{.} \label{57}%
\end{equation}

\end{lemma}

\begin{proof}
It is easy to verify the result.
\end{proof}

\begin{lemma}
\label{2106}Let $\lambda \in \left(  0,\frac{1}{2}\right)  $ be any fixed
constant. Then there is a unique constant $x_{1}^{\lambda}\in \left(
-1,0\right)  $ such that $\varphi_{\lambda}^{\prime \prime}\left(
x_{1}^{\lambda}\right)  =0$, and there is a unique $x_{2}^{\lambda}\in \left(
1,+\infty \right)  $ such that $\varphi_{\lambda}^{\prime \prime}\left(
x_{2}^{\lambda}\right)  =0$. Moreover,%
\[
\left \{
\begin{array}
[c]{l}%
\varphi_{\lambda}^{\prime \prime}\left(  x\right)  <0\text{ \ for }x\in \left(
x_{1}^{\lambda},x_{2}^{\lambda}\right)  \text{,}\\
\varphi_{\lambda}^{\prime \prime}\left(  x\right)  >0\text{ \ for }x\in \left(
-\infty,x_{1}^{\lambda}\right)  \cup \left(  x_{2}^{\lambda},+\infty \right)
\text{.}%
\end{array}
\right.
\]

\end{lemma}

\begin{lemma}
\label{2107}Let $\lambda \in \left(  0,\frac{1}{2}\right)  $, $\mu_{1}\in \left(
0,+\infty \right)  $ and $\mu_{2}\in \left(  0,+\infty \right)  $ be any fixed
constants, and let $\Psi_{\lambda}$ be defined by (\ref{2301}). Then there is
a unique $z_{1}^{\lambda}\in \left(  -x_{2}^{\lambda},x_{1}^{\lambda}\right)  $
such that $\Psi_{\lambda}^{\prime \prime}\left(  z_{1}^{\lambda}\right)  =0$,
and there is a unique $z_{2}^{\lambda}\in \left(  -x_{1}^{\lambda}%
,x_{2}^{\lambda}\right)  $ such that $\Psi_{\lambda}^{\prime \prime}\left(
z_{2}^{\lambda}\right)  =0$. Moreover,%
\[
\left \{
\begin{array}
[c]{l}%
\Psi_{\lambda}^{\prime \prime}\left(  x\right)  <0\text{ \ for }x\in \left(
z_{1}^{\lambda},z_{2}^{\lambda}\right)  \text{,}\\
\Psi_{\lambda}^{\prime \prime}\left(  x\right)  >0\text{ \ for }x\in \left(
-\infty,z_{1}^{\lambda}\right)  \cup \left(  z_{2}^{\lambda},+\infty \right)
\text{.}%
\end{array}
\right.
\]

\end{lemma}

The proofs of Lemma \ref{2106} and Lemma \ref{2107} are provided in Subsection
\ref{0801} and Subsection \ref{0802} respectively.

\begin{remark}
\label{281}For $\mu_{1}=\mu_{2}=1$, it is obvious that $\varphi_{\lambda
}\left(  x\right)  +\varphi_{\lambda}\left(  -x\right)  $ is an even function.
Then by Lemma \ref{2107}, we have $\varphi_{\lambda}^{\prime \prime}\left(
z^{\lambda}\right)  +\varphi_{\lambda}^{\prime \prime}\left(  -z^{\lambda
}\right)  =0$, where $z^{\lambda}$ is defined by Notation \ref{3.2}.
Moreover,
\[
\left \{
\begin{array}
[c]{l}%
\varphi_{\lambda}^{\prime \prime}\left(  x\right)  +\varphi_{\lambda}%
^{\prime \prime}\left(  -x\right)  >0\text{ \ for }\left \vert x\right \vert
>z^{\lambda},\\
\varphi_{\lambda}^{\prime \prime}\left(  x\right)  +\varphi_{\lambda}%
^{\prime \prime}\left(  -x\right)  <0\text{ \ for }\left \vert x\right \vert
<z^{\lambda}.
\end{array}
\right.
\]

\end{remark}

\begin{proof}
[\textnormal{\textbf{Proof of Theorem \ref{2104}}}]
For $\sigma \in \left[  \sigma_{\lambda
},1\right)  $, we have $\sigma z^{\lambda}\in \left[  -x_{1}^{\lambda
},z^{\lambda}\right)  \subseteq \left[  -x_{1}^{\lambda},x_{2}^{\lambda
}\right)  $, which implies $\varphi_{\lambda}^{\prime \prime}\left(  \sigma
z^{\lambda}\right)  <0$ and $\varphi_{\lambda}^{\prime \prime}\left(  -\sigma
z^{\lambda}\right)  \geq0$ by Lemma \ref{2106}. Then it is easy to check that
$\mu_{1}^{\lambda}>0$ and $\mu_{2}^{\lambda}\geq0$ in (\ref{59}), so
$H_{\lambda,\sigma}$ is positive on $\mathbb{R}$.

From (\ref{59}), direct computation shows that
\begin{equation}
\mu_{2}^{\lambda}\varphi_{\lambda}^{\prime \prime}\left(  \sigma z^{\lambda
}\right)  +\mu_{1}^{\lambda}\varphi_{\lambda}^{\prime \prime}\left(  -\sigma
z^{\lambda}\right)  =0\text{.}\label{62701}%
\end{equation}
Then by Lemma \ref{2106} and Lemma \ref{2107}, for $x\in \left(  -\infty,-\sigma
z^{\lambda}\right)  \cup \left(  \sigma z^{\lambda},+\infty \right)  $,
$H_{\lambda,\sigma}^{\prime \prime}\left(  x\right)  >0$ and $H_{\lambda
,\sigma}\left(  x\right)  $ satisfies ODE\ (\ref{52}). By Lemma \ref{2105}, Lemma \ref{2107} and Remark \ref{281}, for $x\in \left(  -\sigma
z^{\lambda},\sigma z^{\lambda}\right)  $, $H_{\lambda,\sigma}^{\prime \prime
}\left(  x\right)  <0$ and $H_{\lambda,\sigma}\left(  x\right)  $ satisfies
ODE\ (\ref{52}).

Now we verify that $H_{\lambda,\sigma}\in C^{2}$. By (\ref{59}) and the
definition of $z^{\lambda}$ in Notation \ref{3.2}, we can easily check that%
\[
\left \{
\begin{array}
[c]{l}%
\lim_{x\rightarrow \left(  \sigma z^{\lambda}\right)  ^{-}}H_{\lambda,\sigma
}\left(  x\right)  =\lim_{x\rightarrow \left(  \sigma z^{\lambda}\right)  ^{+}%
}H_{\lambda,\sigma}\left(  x\right)  \text{,}\\
\lim_{x\rightarrow \left(  \sigma z^{\lambda}\right)  ^{-}}H_{\lambda,\sigma
}^{\prime \prime}\left(  x\right)  =\lim_{x\rightarrow \left(  \sigma
z^{\lambda}\right)  ^{+}}H_{\lambda,\sigma}^{\prime \prime}\left(  x\right)
=0\text{.}%
\end{array}
\right.
\]
Thus $\lim_{x\rightarrow \left(  \sigma z^{\lambda}\right)  ^{-}}%
H_{\lambda,\sigma}^{\prime}\left(  x\right)  =\lim_{x\rightarrow \left(  \sigma
z^{\lambda}\right)  ^{+}}H_{\lambda,\sigma}^{\prime}\left(  x\right)  $ by ODE
(\ref{52}). Since $H_{\lambda,\sigma}$ in (\ref{58}) is obviously even, we
obtain $H_{\lambda,\sigma}\in C^{2}$.
\end{proof}

Now let $\sigma \in \left(  0,1\right)  $ be fixed, we want to find suitable
$\lambda \in \left(  0,\frac{1}{2}\right)  $ such that ODE (\ref{52}) has
positive $C^{2}$ solutions. So we give the following properties of
$\sigma_{\lambda}$.

\begin{proposition}
\label{2109}Let $\sigma_{\lambda}$ be defined by Notation \ref{3.4}. Then
$\sigma_{\lambda}$ is continuous and strictly increasing in $\lambda \in \left(
0,\frac{1}{2}\right)  $, and
\[
\lim_{\lambda \rightarrow0}\sigma_{\lambda}=0,\lim_{\lambda \rightarrow \frac
{1}{2}}\sigma_{\lambda}=1.
\]

\end{proposition}

In order to prove Proposition \ref{2109}, we need the following lemmas.

\begin{lemma}
\label{2303}Let $x_{1}^{\lambda}$ be defined by Notation \ref{3.2}. Then
$x_{1}^{\lambda}$ is continuous and strictly decreasing in $\lambda \in \left(
0,\frac{1}{2}\right)  $, and%
\begin{equation}
\lim_{\lambda \rightarrow0}x_{1}^{\lambda}=0,\lim_{\lambda \rightarrow \frac
{1}{2}}x_{1}^{\lambda}=-1\text{.}\label{3001}%
\end{equation}

\end{lemma}

\begin{lemma}
\label{1801}Let $z^{\lambda}$ be defined by Notation \ref{3.2}. Then
$z^{\lambda}$ is continuous and strictly decreasing in $\lambda \in \left(
0,\frac{1}{2}\right)  $, and%
\[
\lim_{\lambda \rightarrow \frac{1}{2}}z^{\lambda}=1\text{.}%
\]

\end{lemma}

The proofs of Lemma \ref{2303} and Lemma \ref{1801} are provided in Subsection
\ref{1501} and Subsection \ref{0803} respectively.

\begin{proof}
[\textnormal{\textbf{Proof of Proposition \ref{2109}}}]According to Notation \ref{3.4}, the
proposition is directly proved by Lemma \ref{2303} and Lemma \ref{1801}.
\end{proof}

\begin{notation}
\label{3.13}Let $\sigma \in \left(  0,1\right)  $ be any fixed constant. Let
$\lambda_{\sigma}\in \left(  0,\frac{1}{2}\right)  $ denote the unique constant
satisfying
\[
\sigma_{\lambda_{\sigma}}=\sigma \text{,}%
\]
where $\sigma_{\lambda_{\sigma}}$ is defined by Notation \ref{3.4}.
\end{notation}

According to Proposition \ref{2109}, $\lambda_{\sigma}$ is well-defined.
Moreover, we have the following corollary by Proposition \ref{2109}.

\begin{corollary}
\label{2309}Let $\sigma \in \left(  0,1\right)  $ be any fixed constant and let
$\lambda_{\sigma}$ be defined as above. Then $\lambda_{\sigma}$ is strictly
increasing in $\sigma \in \left(  0,1\right)  $, and
\[
\lim_{\sigma \rightarrow0}\lambda_{\sigma}=0,\lim_{\sigma \rightarrow1}%
\lambda_{\sigma}=\frac{1}{2}.
\]

\end{corollary}

We can now state the second main result of explicit positive solutions to ODE
(\ref{52}).

\begin{theorem}
\label{2112}Let $\sigma \in \left(  0,1\right)  $ be any fixed constant. Then
for each $\lambda \in \left(  0,\lambda_{\sigma}\right]  $, $H_{\lambda,\sigma}$
is the $C^{2}$ positive solution to ODE (\ref{52}), where $\lambda_{\sigma}$
is defined by Notation \ref{3.13}\ and $H_{\lambda,\sigma}$ is defined by
(\ref{58}).
\end{theorem}

\begin{proof}
For each $\lambda \in \left(  0,\lambda_{\sigma}\right]  $, by Theorem
\ref{2104}, we know that there exists a $\sigma_{\lambda}$ such that for every
$\sigma^{\prime}\in \left[  \sigma_{\lambda},1\right)  $, ODE (\ref{52}) has
positive solutions. Since $\lambda \leq \lambda_{\sigma}$, we have
$\sigma_{\lambda}\leq \sigma_{\lambda_{\sigma}}=\sigma$ by Proposition
\ref{2109}. Thus $H_{\lambda,\sigma}$ is the $C^{2}$ positive solution to ODE
(\ref{52}), which completes the proof.
\end{proof}

Next theorem is the\ third main result in this section, which gives the
explicit positive solutions to $G$-heat equation (\ref{50}) for each fixed
$\sigma \in \left(  0,1\right)  $.

\begin{theorem}
\label{2113}Let $\sigma \in \left(  0,1\right)  $ be any fixed constant. Then
for each $\lambda \in \left(  0,\lambda_{\sigma}\right]  $,
\[
u^{H_{\lambda,\sigma}}\left(  t,x\right)  =\left(  1+t\right)  ^{-\lambda
}H_{\lambda,\sigma}\left(  \frac{x}{\sqrt{1+t}}\right)
\]
is the positive solution to $G$-heat equation (\ref{50}) with initial
condition $u\left(  0,x\right)  =H_{\lambda,\sigma}\left(  x\right)  $, where
$H_{\lambda,\sigma}$ is defined by (\ref{58}).
\end{theorem}

\begin{proof}
The proof is straightforward according to Theorem \ref{2102} and Theorem \ref{2112}.
\end{proof}

\begin{remark}
\label{2114}For every $\lambda \in \left(  \lambda_{\sigma},\frac{1}{2}\right)
$, there is no positive $C^{2}$ solutions to ODE (\ref{52}) and $G$-heat
equation (\ref{50}), which will be proved by Proposition \ref{2118}.
\end{remark}

The following proposition gives an estimation of $\lambda_{\sigma}$ for each
fixed $\sigma \in \left(  0,1\right)  $.

\begin{proposition}
\label{2115}Let $\sigma \in \left(  0,1\right)  $ be any fixed constant, and let $\lambda_{\sigma}$ be defined by Notation \ref{3.13}. Then
\[
\lambda_{\sigma}>\frac{\sigma^{2}}{2}.
\]
\end{proposition}

\begin{proof}
By (\ref{3}), we have
\begin{align*}
\varphi_{\lambda_{\sigma}}^{\prime \prime}\left(  x\right)   &  =\int
_{0}^{+\infty}y^{-2\lambda_{\sigma}}\left[  -2\lambda_{\sigma}-x\left(
y-x\right)  \right]  \exp \left \{  -\frac{\left(  y-x\right)  ^{2}}{2}\right \}
dy\\
&  =\int_{-x}^{+\infty}\left(  t+x\right)  ^{-2\lambda_{\sigma}}\left(
-2\lambda_{\sigma}-xt\right)  \exp \left \{  -\frac{t^{2}}{2}\right \}  dt\text{,
}%
\end{align*}
where we use the substitution $t=y-x$ in the second equality.

We claim that
\begin{equation}
\lambda_{\sigma}>\frac{1}{2}\left(  x_{1}^{\lambda_{\sigma}}\right)
^{2}\text{.}\label{1001}%
\end{equation}
Otherwise $-2\lambda_{\sigma}-x_{1}^{\lambda_{\sigma}}\cdot \left(
-x_{1}^{\lambda_{\sigma}}\right)  \geq0$, then it is easy to check that
$\left(  t+x_{1}^{\lambda_{\sigma}}\right)  ^{-2\lambda_{\sigma}}\left(
-2\lambda_{\sigma}-x_{1}^{\lambda_{\sigma}}t\right)  \exp \left \{
-t^{2}/2\right \}  >0$ for $t\in \left(  -x_{1}^{\lambda_{\sigma}}%
,+\infty \right)  $ because $x_{1}^{\lambda_{\sigma}}<0$. Thus $\varphi_{\lambda_{\sigma}}^{\prime \prime}\left(
x_{1}^{\lambda_{\sigma}}\right)  >0$, which is contrary to the definition of
$x_{1}^{\lambda_{\sigma}}$ in Notation \ref{3.2}.

Lemma \ref{1801} shows that $z^{\lambda_{\sigma}}>1$. Then from (\ref{1001})
we obtain
\begin{equation}
\lambda_{\sigma}>\frac{\left(  x_{1}^{\lambda_{\sigma}}\right)  ^{2}}{2\left(
z^{\lambda_{\sigma}}\right)  ^{2}}\text{.}\label{1002}%
\end{equation}
It follows from Notation \ref{3.4} and Notation \ref{3.13} that%
\[
\frac{\left(  x_{1}^{\lambda_{\sigma}}\right)  ^{2}}{2\left(  z^{\lambda
_{\sigma}}\right)  ^{2}}=\frac{1}{2}\left(  -\frac{x_{1}^{\lambda_{\sigma}}%
}{z^{\lambda_{\sigma}}}\right)  ^{2}=\frac{\sigma_{\lambda_{\sigma}}^{2}}%
{2}=\frac{\sigma^{2}}{2}\text{,}%
\]
which implies $\lambda_{\sigma}>\sigma^{2}/2$ by (\ref{1002}).
\end{proof}

\section{The Application to $G$-capacities\label{s4}}

In this section, we will give the sharp order of $c_{\sigma}\left(  \left[
-\varepsilon,\varepsilon \right]  \right)  $ when $\varepsilon \rightarrow0$.

For each fixed $\sigma \in \left(  0,1\right)  $, by the definition of
$\lambda_{\sigma}$ in Notation \ref{3.13} and $\sigma_{\lambda}$ in Notation
\ref{3.4}, define%

\[
P_{\sigma}\left(  x\right)  :=H_{\lambda_{\sigma},\sigma}\left(  x\right)
=\left \{
\begin{array}
[c]{lll}%
\mu_{1}^{\lambda_{\sigma}}\varphi_{\lambda_{\sigma}}\left(  x\right)  , & \,
\, \, \, \, \, & x\leq x_{1}^{\lambda_{\sigma}},\\
\varphi_{\lambda_{\sigma}}\left(  \frac{x}{\sigma}\right)  +\varphi
_{\lambda_{\sigma}}\left(  -\frac{x}{\sigma}\right)  , &  & x_{1}%
^{\lambda_{\sigma}}<x<-x_{1}^{\lambda_{\sigma}},\\
\mu_{1}^{\lambda_{\sigma}}\varphi_{\lambda_{\sigma}}\left(  -x\right)  , &  &
x\geq-x_{1}^{\lambda_{\sigma}},
\end{array}
\right.
\]
where
\[
\mu_{1}^{\lambda_{\sigma}}=\left[  \varphi_{\lambda_{\sigma}}\left(
x_{1}^{\lambda_{\sigma}}\right)  \right]  ^{-1}\cdot \left[  \varphi
_{\lambda_{\sigma}}\left(  \frac{x_{1}^{\lambda_{\sigma}}}{\sigma}\right)
+\varphi_{\lambda_{\sigma}}\left(  -\frac{x_{1}^{\lambda_{\sigma}}}{\sigma
}\right)  \right]  >0\text{,}%
\]
$\varphi_{\lambda_{\sigma}}$, $H_{\lambda_{\sigma},\sigma}$ and $x_{1}^{\lambda_{\sigma}}$ are defined by (\ref{5333}), (\ref{58}) and Notation \ref{3.2} respectively. By Theorem
\ref{2112}, $P_{\sigma}$ is a positive $C^{2}$ solution to ODE\ (\ref{52}).
Now we give the properties of $P_{\sigma}\left( \cdot \right)$.

\begin{lemma}
\label{2116}Let $\sigma \in \left(  0,1\right)  $ be any fixed constant. Then
$P_{\sigma}$ is an even function, and $P_{\sigma}\left(  x\right)  $ is
strictly decreasing in $x\in \left(  0,+\infty \right)  $. Moreover,
\begin{equation}
P_{\sigma}\left(  x\right)  <\frac{\mu_{1}^{\sigma}}{2}P_{\sigma}\left(
0\right)  \exp \left \{  -\frac{x^{2}}{2}\right \}  \text{ \ for }x>-x_{1}%
^{\lambda}, \label{61801}%
\end{equation}
and
\[
\lim_{x\rightarrow+\infty}P_{\sigma}\left(  x\right)  =0.
\]

\end{lemma}

\begin{proof}
It is clear that $P_{\sigma}$ is an even function.

By Lemma \ref{2107} we know
that $P_{\sigma}^{\prime \prime}\left(  x\right)  <0$ for $x\in \left[
0,-x_{1}^{\lambda_{\sigma}}\right)  $, and it is clear that $P_{\sigma
}^{\prime}\left(  0\right)  =0$. Then $P_{\sigma}^{\prime}\left(  x\right)
<0$ for$\ x\in \left(  0,-x_{1}^{\lambda_{\sigma}}\right)  $, so $P_{\sigma}$
is strictly decreasing in $x\in \left(  0,-x_{1}^{\lambda_{\sigma}}\right)  $.
By Lemma \ref{2106}, we know that $P_{\sigma}^{\prime \prime}\left(  x\right)
\geq0$ for $x\in \left[  -x_{1}^{\lambda_{\sigma}},+\infty \right)  $. Since
$P_{\sigma}$ is the solution to ODE\ (\ref{52}), for $x\in \left[
-x_{1}^{\lambda_{\sigma}},+\infty \right)  \subseteq \left(  0,+\infty \right)
$, we have
\[
P_{\sigma}^{\prime}\left(  x\right)  =\frac{-2\lambda P_{\sigma}\left(
x\right)  -P_{\sigma}^{\prime \prime}\left(  x\right)  }{x}<0\text{.}%
\]
Therefore, $P_{\sigma}$ is strictly decreasing in $x\in \left[  -x_{1}%
^{\lambda_{\sigma}},+\infty \right)  $.

One can easily check that for $x\in \left(  -x_{1}^{\lambda_{\sigma}}%
,+\infty \right)  $,
\[%
\begin{array}
[c]{rl}%
P_{\sigma}\left(  x\right)  = & \mu_{1}^{\lambda_{\sigma}}\int_{0}^{+\infty
}y^{-2\lambda_{\sigma}}\exp \left \{  -\frac{\left(  y+x\right)  ^{2}}%
{2}\right \}  dy\\
< & \mu_{1}^{\lambda_{\sigma}}\exp \left \{  -\frac{x^{2}}{2}\right \}  \int
_{0}^{+\infty}y^{-2\lambda_{\sigma}}\exp \left \{  -\frac{y^{2}}{2}\right \}
dy\\
= & \frac{\mu_{1}^{\sigma}}{2}P_{\sigma}\left(  0\right)  \exp \left \{
-\frac{x^{2}}{2}\right \}  .
\end{array}
\]
It follows that $\lim_{x\rightarrow+\infty}P_{\sigma}\left(  x\right)  =0$.
\end{proof}

Now we state the main result of this section.

\begin{theorem}
\label{2117}Let $\sigma \in \left(  0,1\right)  $ be any fixed constant, and let $\lambda_{\sigma}$ be defined by Notation \ref{3.13}. Then
for every $\varepsilon \in \left(  0,1\right)  $,
\begin{equation}
c_{\sigma}\left(  \left[  -\varepsilon,\varepsilon \right]  \right)  \leq
\frac{P_{\sigma}\left(  0\right)  }{P_{\sigma}\left(  1\right)  }%
\varepsilon^{2\lambda_{\sigma}}\text{,} \label{2306}%
\end{equation}
and
\begin{equation}
c_{\sigma}\left(  \left[  -\varepsilon \sqrt{2\ln \left(  r_{\sigma}%
\varepsilon^{-2\lambda_{\sigma}}\right)  },\varepsilon \sqrt{2\ln \left(
r_{\sigma}\varepsilon^{-2\lambda_{\sigma}}\right)  }\right]  \right)
\geq2^{-\lambda_{\sigma}-1}\varepsilon^{2\lambda_{\sigma}}\text{,}
\label{2307}%
\end{equation}
where $r_{\sigma}=4\left[  2\vee \left(  {\mu_{1}^{\lambda_{\sigma}}}%
/{2}\right)  \right]  $.
\end{theorem}

\begin{proof}
By Lemma \ref{2116}, we have $I_{\left[  -1,1\right]  }\left(  x\right)  \leq
P_{\sigma}\left(  x\right)  /P_{\sigma}\left(  1\right)  $. According to
(\ref{62103}) and Definition \ref{62104}, we obtain that%
\begin{equation}
c_{\sigma}\left(  \left[  -\varepsilon,\varepsilon \right]  \right)
=c_{\sigma}^{\frac{1}{\varepsilon^{2}}}\left(  \left[  -1,1\right]  \right)
\leq \mathbb{\hat{E}}_{\sigma}^{\frac{1}{\varepsilon^{2}}}\left[  \frac
{1}{P_{\sigma}\left(  1\right)  }P_{\sigma}\left(  \cdot \right)  \right]
.\label{62106}%
\end{equation}
Then, by the positive homogeneity of $\mathbb{\hat{E}}_{\sigma}^{t}$ and
Definition \ref{62105}, we have
\begin{equation}
\mathbb{\hat{E}}_{\sigma}^{\frac{1}{\varepsilon^{2}}}\left[  \frac
{1}{P_{\sigma}\left(  1\right)  }P_{\sigma}\left(  \cdot \right)  \right]
=\frac{1}{P_{\sigma}\left(  1\right)  }\mathbb{\hat{E}}_{\sigma}^{\frac
{1}{\varepsilon^{2}}}\left[  P_{\sigma}\left(  \cdot \right)  \right]
=\frac{1}{P_{\sigma}\left(  1\right)  }u^{P_{\sigma}\left(  \cdot \right)
}\left(  \frac{1}{\varepsilon^{2}},0\right)  .
\label{3003}%
\end{equation}
It follows by Theorem \ref{2113} that
\begin{equation}
u^{P_{\sigma}\left(  \cdot \right)  }\left(  \frac{1}{\varepsilon^{2}%
},0\right)  =\left(  1+\frac{1}{\varepsilon^{2}}\right)  ^{-\lambda_{\sigma}%
}P_{\sigma}\left(  0\right)  \leq \varepsilon^{2\lambda_{\sigma}}P_{\sigma
}\left(  0\right)  .\label{62107}%
\end{equation}
Thus we get (\ref{2306}) by (\ref{62106})-(\ref{62107}).

Set
\[
a=\frac{1}{2}\left(  1+\frac{1}{\varepsilon^{2}}\right)  ^{-\lambda_{\sigma}%
}P_{\sigma}\left(  0\right)  .
\]
It is clear that $a\rightarrow0$ when $\varepsilon \rightarrow0$, and
$a\in \left(  0,P_{\sigma}\left(  0\right)  \right)  $. From Lemma \ref{2116}
we know that there exists a unique $l_{a}>0$ such that $P_{\sigma}\left(
l_{a}\right)  =P_{\sigma}\left(  -l_{a}\right)  =a$, and $I_{\left[
-l_{a},l_{a}\right]  }\left(  x\right)  \geq \left[  P_{\sigma}\left(
x\right)  -a\right]  /\left[  P_{\sigma}\left(  0\right)  -a\right]  $.
According to (\ref{62103}) and Remark \ref{62108}, we have%
\begin{equation}
c_{\sigma}\left(  \left[  -\varepsilon l_{a},\varepsilon l_{a}\right]
\right)  =c_{\sigma}^{\frac{1}{\varepsilon^{2}}}\left(  \left[  -l_{a}%
,l_{a}\right]  \right)  \geq \mathbb{\hat{E}}_{\sigma}^{\frac{1}{\varepsilon
^{2}}}\left[  \frac{P_{\sigma}\left(  \cdot \right)  -a}{P_{\sigma}\left(
0\right)  -a}\right]  .\label{62110}%
\end{equation}
Then, by the properties of $\mathbb{\hat{E}}_{\sigma}^{t}$ and Definition
\ref{62105}, we get that%
\begin{equation}
\mathbb{\hat{E}}_{\sigma}^{\frac{1}{\varepsilon^{2}}}\left[  \frac{P_{\sigma
}\left(  \cdot \right)  -a}{P_{\sigma}\left(  0\right)  -a}\right]  =\frac
{1}{P_{\sigma}\left(  0\right)  -a}\left(  \mathbb{\hat{E}}_{\sigma}^{\frac
{1}{\varepsilon^{2}}}\left[  P_{\sigma}\left(  \cdot \right)  \right]
-a\right)  =\frac{1}{P_{\sigma}\left(  0\right)  -a}\left(  u^{P_{\sigma
}\left(  \cdot \right)  }\left(  \frac{1}{\varepsilon^{2}},0\right)  -a\right)
.\label{62109}%
\end{equation}
Substituting $a$ into (\ref{62109}), by Theorem \ref{2113}, we can easily
check that%
\begin{equation}
\frac{1}{P_{\sigma}\left(  0\right)  -a}\left(  u^{P_{\sigma}\left(
\cdot \right)  }\left(  \frac{1}{\varepsilon^{2}},0\right)  -a\right)
=\frac{\varepsilon^{2\lambda_{\sigma}}}{2\left(  \varepsilon^{2}+1\right)
^{\lambda_{\sigma}}-\varepsilon^{2\lambda_{\sigma}}}\geq \frac{\varepsilon
^{2\lambda_{\sigma}}}{2\left(  \varepsilon^{2}+1\right)  ^{\lambda_{\sigma}}%
}\geq2^{-\lambda_{\sigma}-1}\varepsilon^{2\lambda_{\sigma}}.\label{62111}%
\end{equation}
By (\ref{62110})-(\ref{62111}), we conclude that%
\begin{equation}
c_{\sigma}\left(  \left[  -\varepsilon l_{a},\varepsilon l_{a}\right]
\right)  \geq2^{-\lambda_{\sigma}-1}\varepsilon^{2\lambda_{\sigma}%
}.\label{62113}%
\end{equation}

We now give an estimation of $l_{a}$ with respect to $\sigma$ and
$\varepsilon$. Let $\delta=P_{\sigma}\left(  0\right)  \left[  2\vee \left(
\mu_{1}^{\lambda_{\sigma}}/2\right)  \right]  $. Since $a\in \left(
0,P_{\sigma}\left(  0\right)  \right)  $, we can easily check that there
exists $x_{a}>1$ such $\exp \left \{  -x_{a}^{2}/2\right \}  \delta=a$. Then
simple computation shows that%
\begin{equation}
1<x_{a}=\sqrt{-2\ln \left[  \frac{P_{\sigma}\left(  0\right)  }{2\delta}\left(
\frac{1+\varepsilon^{2}}{\varepsilon^{2}}\right)  ^{-\lambda_{\sigma}}\right]
}<\sqrt{2\ln \left(  r_{\sigma}\varepsilon^{-2\lambda_{\sigma}}\right)
},\label{62112}%
\end{equation}
where $r_{\sigma}=4\left[  2\vee \left(  \mu_{1}^{\lambda_{\sigma}}/2\right)
\right]  $. From (\ref{61801}), we get that
\[
P_{\sigma}\left(  x_{a}\right)  <\frac{\mu_{1}^{\sigma}}{2}P_{\sigma}\left(
0\right)  \exp \left \{  -\frac{x_{a}^{2}}{2}\right \}  \leq \exp \left \{
-\frac{x_{a}^{2}}{2}\right \}  \delta=a=P_{\sigma}\left(  l_{a}\right)  .
\]
It follows that $l_{a}<x_{a}$ by Lemma \ref{2116}. Then we have $l_{a}%
<\sqrt{2\ln \left(  r_{\sigma}\varepsilon^{-2\lambda_{\sigma}}\right)  }$ by
(\ref{62112}), which implies (\ref{2307}) by (\ref{62113}).
\end{proof}

\begin{corollary}
\label{62114}For $\lambda \in \left(  \lambda_{\sigma},\frac{1}{2}\right)  $,
\[
\liminf \limits_{\varepsilon \rightarrow0}\frac{c_{\sigma}\left(  \left[
-\varepsilon,\varepsilon \right]  \right)  }{\varepsilon^{2\lambda}}%
=+\infty \text{.}%
\]

\end{corollary}

\begin{proof}
Set $\varepsilon^{\prime}=\varepsilon \sqrt{2\ln \left(  r_{\sigma}%
\varepsilon^{-2\lambda_{\sigma}}\right)  }$. It is clear that $\varepsilon
^{\prime}>\varepsilon$ and $\varepsilon^{\prime}\rightarrow0$ when
$\varepsilon \rightarrow0$. For every $\varepsilon^{\prime}\in \left(
0,1\right)  $, from (\ref{2307}), we can easily check that
\begin{align*}
\liminf \limits_{\varepsilon^{\prime}\rightarrow0}\frac{c_{\sigma}\left(
\left[  -\varepsilon^{\prime},\varepsilon^{\prime}\right]  \right)  }{\left(
\varepsilon^{\prime}\right)  ^{2\lambda}} &  =\liminf \limits_{\varepsilon
^{\prime}\rightarrow0}\frac{c_{\sigma}\left(  \left[  -\varepsilon^{\prime
},\varepsilon^{\prime}\right]  \right)  }{\varepsilon^{2\lambda_{\sigma}%
}\varepsilon^{2\left(  \lambda-\lambda_{\sigma}\right)  }\left[  2\ln \left(
r_{\sigma}\varepsilon^{-2\lambda_{\sigma}}\right)  \right]  ^{\lambda}}\\
&  \geq \liminf \limits_{\varepsilon \rightarrow0}\frac{2^{-\lambda_{\sigma}-1}%
}{\varepsilon^{2\left(  \lambda-\lambda_{\sigma}\right)  }\left[  2\ln \left(
r_{\sigma}\varepsilon^{-2\lambda_{\sigma}}\right)  \right]  ^{\lambda}}\\
&  =+\infty \text{,}%
\end{align*}
which completes the proof.
\end{proof}

\begin{remark}
For any $m\in \left(  0,+\infty \right)  $ and each fixed $t>m^{2}$, by
(\ref{62103}) and Theorem \ref{2117}, we have
\[
c_{\sigma}^{t}\left(  \left[  -m,m\right]  \right)  =c_{\sigma}\left(  \left[
-\frac{m}{\sqrt{t}},\frac{m}{\sqrt{t}}\right]  \right)  \leq \frac{P_{\sigma
}\left(  0\right)  }{P_{\sigma}\left(  1\right)  }\left(  \frac{m}{\sqrt{t}%
}\right)  ^{2\lambda_{\sigma}}=\frac{P_{\sigma}\left(  0\right)
m^{2\lambda_{\sigma}}}{P_{\sigma}\left(  1\right)  }t^{-\lambda_{\sigma}%
}\text{.}%
\]

\end{remark}

Now we show that $\lambda_{\sigma}$ is the maximal value such that ODE
(\ref{52}) and $G$-heat equation (\ref{50}) satisfying (\ref{54}) have
positive solutions.

\begin{proposition}
\label{2118}Let $\sigma \in \left(  0,1\right)  $ be any fixed constant. Then
for every $\lambda \in \left(  \lambda_{\sigma},\frac{1}{2}\right)  $, there is
no positive solutions to ODE (\ref{52}) and $G$-heat equation (\ref{50}) with
initial condition $u\left(  0,x\right)  =H\left(  x\right)  $ satisfying
(\ref{54}).
\end{proposition}

\begin{proof}
To obtain a contradiction, we suppose that there exists a $\lambda
>\lambda_{\sigma}$ such that ODE\ (\ref{52}) has a positive $C^{2}$ solution
$\tilde{P}$. Then there exists a constant$\ K>0$ such that
\[
I_{\left[  -1,1\right]  }\left(  x\right)  \leq K\tilde{P}\left(  x\right)
\text{.}%
\]
for $x\in%
\mathbb{R}
$. Analysis similar to that in the proof of Theorem \ref{2117} shows that%
\begin{equation}
\limsup_{\varepsilon \rightarrow0}\frac{c_{\sigma}\left(  \left[
-\varepsilon,\varepsilon \right]  \right)  }{\varepsilon^{2\lambda}}<+\infty,
\label{1212}%
\end{equation}
which contradicts Corollary \ref{62114}.

Therefore, by Theorem \ref{2102}
and Remark \ref{2101}, there is no positive solutions to $G$-heat equation
(\ref{50}) with initial condition $u\left(  0,x\right)  =H\left(  x\right)  $ satisfying
(\ref{54}).
\end{proof}

\section{Technical Proofs\label{s5}}

We first give some equalities which are useful in the following proofs.

For each $\mu_{1},\mu_{2}\in%
\mathbb{R}
$, the general solution $\Psi_{\lambda}\left(  \cdot \right)  $\ in
(\ref{2301}) to ODE (\ref{55}) satisfies%
\begin{equation}
\Psi_{\lambda}^{\prime \prime}\left(  x\right)  =-2\lambda \Psi_{\lambda}\left(
x\right)  -x\Psi_{\lambda}^{\prime}\left(  x\right)  \text{.}\label{1}%
\end{equation}
Taking the derivation of above equation, we obtain
\begin{equation}
\Psi_{\lambda}^{\prime \prime \prime}\left(  x\right)  =\left(  -2\lambda
-1\right)  \Psi_{\lambda}^{\prime}\left(  x\right)  -x\Psi_{\lambda}%
^{\prime \prime}\left(  x\right)  \text{.}\label{10}%
\end{equation}
Combining (\ref{1}) and (\ref{10}), we have%

\begin{equation}
\Psi_{\lambda}^{\prime \prime \prime}\left(  x\right)  =\left(  x^{2}%
-2\lambda-1\right)  \Psi_{\lambda}^{\prime}\left(  x\right)  +2\lambda
x\Psi_{\lambda}\left(  x\right)  \label{9}%
\end{equation}
and%

\begin{equation}
\Psi_{\lambda}^{\prime \prime \prime}\left(  x\right)  =\frac{2\lambda \left(
2\lambda+1\right)  }{x}\Psi_{\lambda}\left(  x\right)  -\frac{x^{2}%
-2\lambda-1}{x}\Psi_{\lambda}^{\prime \prime}\left(  x\right)  \text{.}
\label{8}%
\end{equation}
Obviously, $\varphi_{\lambda}\left(  \cdot \right)  $ defined by (\ref{5333})
also satisfies (\ref{1})-(\ref{8}).

\subsection{\textbf{Proof of Lemma \ref{2106}\label{0801}}}

The proof is divided into three steps.

\textbf{Step 1:} We first show that there exists $x_{1}\in \left(  -1,0\right)
$ and $x_{2}\in \left(  1,+\infty \right)  $ such that $\varphi_{\lambda
}^{\prime \prime}\left(  x_{1}\right)  =\varphi_{\lambda}^{\prime \prime}\left(
x_{2}\right)  =0$.

Since $\left(  y+1\right)  ^{2}-1$ is positive for $y\in \left(  0,+\infty
\right)  $, by (\ref{2402}), it is easily seen that
\begin{equation}
\varphi_{\lambda}^{\prime \prime}\left(  -1\right)  =\int_{0}^{+\infty
}y^{-2\lambda}\left[  \left(  y+1\right)  ^{2}-1\right]  \exp \left \{
-\frac{y^{2}}{2}\right \}  dy>0\text{.}\label{1802}%
\end{equation}
From (\ref{3}) we have
\[
\varphi_{\lambda}^{\prime \prime}\left(  0\right)  =-2\lambda \int_{0}^{+\infty
}y^{-2\lambda}\exp \left \{  -\frac{y^{2}}{2}\right \}  dy<0\text{.}%
\]
So there exists a $x_{1}\in \left(  -1,0\right)  $ such that $\varphi_{\lambda
}^{\prime \prime}\left(  x_{1}\right)  =0$ by the continuity of $\varphi
_{\lambda}^{\prime \prime}$.

We next show that $\varphi_{\lambda}^{\prime \prime}\left(  1\right)  <0$. It
is sufficient to prove that $L\left(  \lambda \right)  :=\varphi_{\lambda
}^{\prime \prime}\left(  1\right)  +\varphi_{\lambda}^{\prime \prime}\left(
-1\right)  <0$ for every $\lambda \in \left(  0,\frac{1}{2}\right)  $ according
to (\ref{1802}). By (\ref{2402}), the substitution $t=y-1$ and $s=-t$\ enables
us to get that%
\[%
\begin{array}
[c]{rl}%
\varphi_{\lambda}^{\prime \prime}\left(  1\right)  = & \int_{-1}^{+\infty
}\left(  t+1\right)  ^{-2\lambda}\left(  t^{2}-1\right)  \exp \left \{
-\frac{t^{2}}{2}\right \}  dt\\
= & \int_{-1}^{1}\left(  t+1\right)  ^{-2\lambda}\left(  t^{2}-1\right)
\exp \left \{  -\frac{t^{2}}{2}\right \}  dt+\int_{1}^{+\infty}\left(
t+1\right)  ^{-2\lambda}\left(  t^{2}-1\right)  \exp \left \{  -\frac{t^{2}}%
{2}\right \}  dt\\
= & \int_{0}^{1}\left(  t+1\right)  ^{-2\lambda}\left(  t^{2}-1\right)
\exp \left \{  -\frac{t^{2}}{2}\right \}  dt+\int_{0}^{1}\left(  1-s\right)
^{-2\lambda}\left(  s^{2}-1\right)  \exp \left \{  -\frac{s^{2}}{2}\right \}
ds\\
& +\int_{1}^{+\infty}\left(  t+1\right)  ^{-2\lambda}\left(  t^{2}-1\right)
\exp \left \{  -\frac{t^{2}}{2}\right \}  dt\text{.}%
\end{array}
\]
Substituting $t=y+1$, we have%
\[%
\begin{array}
[c]{rl}%
\varphi_{\lambda}^{\prime \prime}\left(  -1\right)  = & \int_{1}^{+\infty
}\left(  t-1\right)  ^{-2\lambda}\left(  t^{2}-1\right)  \exp \left \{
-\frac{t^{2}}{2}\right \}  dt\text{.}%
\end{array}
\]
Obviously, we have%
\[
L\left(  \lambda \right)  =g\left(  \lambda \right)  +h\left(  \lambda \right)
\text{,}%
\]
where%
\begin{align*}
g\left(  \lambda \right)   &  :=\int_{1}^{+\infty}\left[  \left(  t+1\right)
^{-2\lambda}+\left(  t-1\right)  ^{-2\lambda}\right]  \left(  t^{2}-1\right)
\exp \left \{  -\frac{t^{2}}{2}\right \}  dt\text{,}\\
h\left(  \lambda \right)   &  :=\int_{0}^{1}\left[  \left(  t+1\right)
^{-2\lambda}+\left(  1-t\right)  ^{-2\lambda}\right]  \left(  t^{2}-1\right)
\exp \left \{  -\frac{t^{2}}{2}\right \}  dt\text{.}%
\end{align*}
Direct computation shows that%
\[
g^{\prime \prime}\left(  \lambda \right)  =4\int_{1}^{+\infty}\left[  \left(
t-1\right)  ^{-2\lambda}\left[  \ln \left(  t-1\right)  \right]  ^{2}+\left(
t+1\right)  ^{-2\lambda}\left[  \ln \left(  t+1\right)  \right]  ^{2}\right]
\left(  t^{2}-1\right)  \exp \left \{  -\frac{t^{2}}{2}\right \}  dt\text{.}%
\]
Obviously $g^{\prime \prime}\left(  \lambda \right)  >0$ for $\lambda \in \left(
0,\frac{1}{2}\right)  $, so $g\left(  \cdot \right)  $ is convex. Thus, noting
that $\lim_{\lambda \rightarrow0}g\left(  \lambda \right)  =\lim_{\lambda
\rightarrow \frac{1}{2}}g\left(  \lambda \right)  =2\exp \left \{  -\frac{1}%
{2}\right \}  $, for $\lambda \in \left( 0,\frac{1}{2}\right)$, we have
\begin{equation}
g\left(  \lambda \right)<\sup_{\lambda \in \left(  0,\frac{1}{2}\right)  }g\left(  \lambda \right)
=2\exp \left \{  -\frac{1}{2}\right \}  \text{.}\label{33}%
\end{equation}
As for $h\left(  \cdot \right)  $, it is easy to check that%
\begin{equation}
h^{\prime}\left(  \lambda \right)  =-2\int_{0}^{1}\left[  \left(  t+1\right)
^{-2\lambda}\ln \left(  t+1\right)  +\left(  1-t\right)  ^{-2\lambda}\ln \left(
1-t\right)  \right]  \left(  t^{2}-1\right)  \exp \left \{  -\frac{t^{2}}%
{2}\right \}  dt\text{.}\label{0804}%
\end{equation}
For $t\in \left(  0,1\right)  $, we have%
\[
0<\ln \left(  1+t\right)  <-\ln \left(  1-t\right)  \text{.}%
\]
and
\[
\left(  t+1\right)  ^{-2\lambda}\ln \left(  1+t\right)  <-\left(  t+1\right)
^{-2\lambda}\ln \left(  1-t\right)  <-\left(  1-t\right)  ^{-2\lambda}%
\ln \left(  1-t\right)  \text{.}%
\]
Then we can easily verify that $h^{\prime}\left(  \lambda \right)  <0$ for
every $\lambda \in \left(  0,\frac{1}{2}\right)  $ by (\ref{0804}), which
implies
\begin{equation}
h\left(  \lambda \right)  <\lim_{\lambda \rightarrow0}h\left(  \lambda \right)
=-2\exp \left \{  -\frac{1}{2}\right \}  \text{.}\label{34}%
\end{equation}
By (\ref{33}) and (\ref{34}), for every $\lambda \in \left(  0,\frac{1}%
{2}\right)  $, we have
\[
L\left(  \lambda \right)  <\sup_{\lambda \in \left(  0,\frac{1}{2}\right)
}g\left(  \lambda \right)  +\lim_{\lambda \rightarrow0}h\left(  \lambda \right)
=0\text{.}%
\]
It follows that
\begin{equation}
\varphi_{\lambda}^{\prime \prime}\left(  1\right)  +\varphi_{\lambda}%
^{\prime \prime}\left(  -1\right)  <0\label{2801}%
\end{equation}
for every $\lambda \in \left(  0,\frac{1}{2}\right)  $, and consequently
$\varphi_{\lambda}^{\prime \prime}\left(  1\right)  <0$.

Now we prove that there exists a $\tilde{x}\in \left(  2,+\infty \right)  $ such
that $\varphi_{\lambda}^{\prime \prime}\left(  \tilde{x}\right)  >0$. We first show that%
\begin{equation}
\lim_{x\rightarrow+\infty}\varphi_{\lambda}^{\prime \prime}\left(  x\right)
=0\text{.}\label{2701}%
\end{equation}
Let $x>2$. By (\ref{2402}), we have%
\begin{equation}%
\begin{array}
[c]{rl}%
\varphi_{\lambda}^{\prime \prime}\left(  x\right)  = & \int_{0}^{+\infty
}y^{-2\lambda}\left[  \left(  y-x\right)  ^{2}-1\right]  \exp \left \{
-\frac{\left(  y-x\right)  ^{2}}{2}\right \}  dy\\
= & \int_{0}^{\frac{x}{2}}y^{-2\lambda}\left[  \left(  y-x\right)
^{2}-1\right]  \exp \left \{  -\frac{\left(  y-x\right)  ^{2}}{2}\right \}
dy+\int_{\frac{x}{2}}^{+\infty}y^{-2\lambda}\left[  \left(  y-x\right)
^{2}-1\right]  \exp \left \{  -\frac{\left(  y-x\right)  ^{2}}{2}\right \}
dy\text{.}%
\end{array}
\label{2702}%
\end{equation}
For the first term of (\ref{2702}), since $\left(  y-x\right)  ^{2}-1>0$ for
$y\in \left(  0,\frac{x}{2}\right)  \subseteq \left(  0,x-1\right)  $, it is
clear that%
\begin{equation}
0<\int_{0}^{\frac{x}{2}}y^{-2\lambda}\left[  \left(  y-x\right)
^{2}-1\right]  \exp \left \{  -\frac{\left(  y-x\right)  ^{2}}{2}\right \}
dy\leq \left(  x^{2}-1\right)  \exp \left \{  -\frac{x^{2}}{8}\right \}  \int
_{0}^{\frac{x}{2}}y^{-2\lambda}dy\rightarrow0,\label{2703}%
\end{equation}
when $x\rightarrow+\infty$. For the second term of (\ref{2702}), noting that
\begin{equation}
\int_{a}^{b}\left(  t^{2}-1\right)  \exp \left \{  -\frac{t^{2}}{2}\right \}
dt=a\exp \left \{  -\frac{a^{2}}{2}\right \}  -b\exp \left \{  -\frac{b^{2}}%
{2}\right \}  \text{,}\label{2705}%
\end{equation}
we can easily check that
\begin{equation}%
\begin{array}
[c]{ll}
& \left \vert \int_{\frac{x}{2}}^{+\infty}y^{-2\lambda}\left[  \left(
y-x\right)  ^{2}-1\right]  \exp \left \{  -\frac{\left(  y-x\right)  ^{2}}%
{2}\right \}  dy\right \vert \\
\leq & \int_{\frac{x}{2}}^{x-1}y^{-2\lambda}\left[  \left(  y-x\right)
^{2}-1\right]  \exp \left \{  -\frac{\left(  y-x\right)  ^{2}}{2}\right \}
dy+\left \vert \int_{x-1}^{x+1}y^{-2\lambda}\left[  \left(  y-x\right)
^{2}-1\right]  \exp \left \{  -\frac{\left(  y-x\right)  ^{2}}{2}\right \}
dy\right \vert \\
& +\int_{x+1}^{+\infty}y^{-2\lambda}\left[  \left(  y-x\right)  ^{2}-1\right]
\exp \left \{  -\frac{\left(  y-x\right)  ^{2}}{2}\right \}  dy\\
\leq & \left(  \frac{x}{2}\right)  ^{-2\lambda}\int_{\frac{x}{2}}^{x-1}\left[
\left(  y-x\right)  ^{2}-1\right]  \exp \left \{  -\frac{\left(  y-x\right)
^{2}}{2}\right \}  dy\\
& -\left(  x-1\right)  ^{-2\lambda}\int_{x-1}^{x+1}\left[  \left(  y-x\right)
^{2}-1\right]  \exp \left \{  -\frac{\left(  y-x\right)  ^{2}}{2}\right \}  dy\\
& +\left(  x+1\right)  ^{-2\lambda}\int_{x+1}^{+\infty}\left[  \left(
y-x\right)  ^{2}-1\right]  \exp \left \{  -\frac{\left(  y-x\right)  ^{2}}%
{2}\right \}  dy\\
\rightarrow & 0\text{,}%
\end{array}
\label{2704}%
\end{equation}
when $x\rightarrow+\infty$. Combining (\ref{2703}) and (\ref{2704})
with\ (\ref{2702}), we finally get (\ref{2701}). Similarly we can prove that
\begin{equation}
\lim_{x\rightarrow+\infty}\varphi_{\lambda}\left(  x\right)  =\int
_{0}^{+\infty}y^{-2\lambda}\exp \left \{  -\frac{\left(  y-x\right)  ^{2}}%
{2}\right \}  dy=0\text{.}\label{2416}%
\end{equation}
We assert that there exists a $\tilde{x}\in \left(  2,+\infty \right)  $ such
that $\varphi^{\prime \prime}\left(  \tilde{x}\right)  >0$. Otherwise
$\varphi_{\lambda}^{\prime \prime}\left(  x\right)  \leq0$ for every
$x\in \left(  2,+\infty \right)  $. Since $\varphi_{\lambda}$ satisfies
ODE\ (\ref{55}), we have $\lim_{x\rightarrow+\infty}\varphi_{\lambda}^{\prime
}\left(  x\right)  =0$ by (\ref{2701}) and (\ref{2416}). According to the
assumption that $\varphi_{\lambda}^{\prime \prime}$ is non-positive, we have
$\varphi_{\lambda}^{\prime}\left(  x\right)  \geq0$ for every $x\in \left(
2,+\infty \right)  $, which contradicts (\ref{2416}) and the fact that
$\varphi_{\lambda}\left(  x\right)  >0$ for every $x\in \left(  2,+\infty
\right)  $. Thus there exists a $\tilde{x}\in \left(  2,+\infty \right)  $ such
that $\varphi_{\lambda}^{\prime \prime}\left(  \tilde{x}\right)  >0$.

According to the above discussion\ and the continuity of $\varphi_{\lambda
}^{\prime \prime}$, we know there exists a $x_{2}\in \left(1,+\infty \right)
$ such that $\varphi_{\lambda}^{\prime \prime}\left(  x_{2}\right)  =0$.

\textbf{Step 2:} We next prove the uniqueness of $x_{1}^{\lambda}$ on $\left(
-1,0\right)  $ and the uniqueness of $x_{2}^{\lambda}$ on $\left(
1,+\infty \right)  $.

Define%
\begin{align*}
x_{1}^{\lambda}  &  :=\sup \left \{  x\in \left(  -1,0\right)  :\varphi_{\lambda
}^{\prime \prime}\left(  x\right)  =0\right \}  \text{,}\\
\, \,x_{2}^{\lambda}  &  :=\inf \left \{  x\in \left(  1,+\infty \right)
:\varphi_{\lambda}^{\prime \prime}\left(  x\right)  =0\right \}  \text{.}%
\end{align*}
Notice that $\varphi_{\lambda}^{\prime \prime}\left(  x_{1}^{\lambda}\right)
=\varphi_{\lambda}^{\prime \prime}\left(  x_{2}^{\lambda}\right)  =0$. Then,
according to Step 1, it is clear that $\varphi_{\lambda}^{\prime \prime}\left(
x\right)  <0$ for $x\in \left(  x_{1}^{\lambda},0\right)  \cup \left(
1,x_{2}^{\lambda}\right)  $ by the continuity of $\varphi_{\lambda}%
^{\prime \prime}$. Now we prove that $\varphi_{\lambda}^{\prime \prime}\left(
x\right)  >0$ for every $x\in \left(  -1,x_{1}^{\lambda}\right)  \cup \left(
x_{2}^{\lambda},+\infty \right)  $, which implies the uniqueness of
$x_{1}^{\lambda}$ on $\left(  -1,0\right)  $ and the uniqueness of
$x_{2}^{\lambda}$ on $\left(  1,+\infty \right)  $.

For $x\in \left(  -1,x_{1}^{\lambda}\right)  \subseteq \left(  -1,0\right)  $,
from (\ref{2402}) we have%
\begin{equation}%
\begin{array}
[c]{cc}%
\varphi_{\lambda}^{\prime \prime}\left(  x\right)  = & \int_{0}^{x+1}%
y^{-2\lambda}\left[  \left(  y-x\right)  ^{2}-1\right]  \exp \left \{
-\frac{\left(  y-x\right)  ^{2}}{2}\right \}  dy+\int_{x+1}^{+\infty
}y^{-2\lambda}\left[  \left(  y-x\right)  ^{2}-1\right]  \exp \left \{
-\frac{\left(  y-x\right)  ^{2}}{2}\right \}  dy\text{.}%
\end{array}
\label{28001}%
\end{equation}
For the first term in (\ref{28001}), since $0<y-x_{1}^{\lambda}<y-x<1$ for
$y\in \left(  0,x+1\right)  $, we have
\begin{equation}%
\begin{array}
[c]{rl}
& \int_{0}^{x+1}y^{-2\lambda}\left[  \left(  y-x\right)  ^{2}-1\right]
\exp \left \{  -\frac{\left(  y-x\right)  ^{2}}{2}\right \}  dy\\
> & \int_{0}^{x+1}y^{-2\lambda}\left[  \left(  y-x_{1}^{\lambda}\right)
^{2}-1\right]  \exp \left \{  -\frac{\left(  y-x_{1}^{\lambda}\right)  ^{2}}%
{2}\right \}  dy\\
= & \int_{0}^{x_{1}^{\lambda}+1}y^{-2\lambda}\left[  \left(  y-x_{1}^{\lambda
}\right)  ^{2}-1\right]  \exp \left \{  -\frac{\left(  y-x_{1}^{\lambda}\right)
^{2}}{2}\right \}  dy\\
& -\int_{x+1}^{x_{1}^{\lambda}+1}y^{-2\lambda}\left[  \left(  y-x_{1}%
^{\lambda}\right)  ^{2}-1\right]  \exp \left \{  -\frac{\left(  y-x_{1}%
^{\lambda}\right)  ^{2}}{2}\right \}  dy\\
> & \int_{0}^{x_{1}^{\lambda}+1}y^{-2\lambda}\left[  \left(  y-x_{1}^{\lambda
}\right)  ^{2}-1\right]  \exp \left \{  -\frac{\left(  y-x_{1}^{\lambda}\right)
^{2}}{2}\right \}  dy\text{.}%
\end{array}
\label{2418}%
\end{equation}
For the second term in (\ref{28001}), it is easy to check that $\left(
t+x\right)  ^{-2\lambda}>\left(  t+x_{1}^{\lambda}\right)  ^{-2\lambda}$ for
$t>1$. Then we have%
\begin{equation}%
\begin{array}
[c]{rl}
& \int_{x+1}^{+\infty}y^{-2\lambda}\left[  \left(  y-x\right)  ^{2}-1\right]
\exp \left \{  -\frac{\left(  y-x\right)  ^{2}}{2}\right \}  dy\\
= & \int_{1}^{+\infty}\left(  t+x\right)  ^{-2\lambda}\left(  t^{2}-1\right)
\exp \left \{  -\frac{t^{2}}{2}\right \}  dt\\
> & \int_{1}^{+\infty}\left(  t+x_{1}^{\lambda}\right)  ^{-2\lambda}\left(
t^{2}-1\right)  \exp \left \{  -\frac{t^{2}}{2}\right \}  dt\\
= & \int_{x_{1}^{\lambda}+1}^{+\infty}s^{-2\lambda}\left[  \left(
s-x_{1}^{\lambda}\right)  ^{2}-1\right]  \exp \left \{  -\frac{\left(
s-x_{1}^{\lambda}\right)  ^{2}}{2}\right \}  ds%
\end{array}
\label{2802}%
\end{equation}
by substitution $t=y-x$ in the first equation\ and $s=t+x_{1}^{\lambda}$ in
the last equation. Combining (\ref{2418}) and (\ref{2802}) with (\ref{28001}),
we get that $\varphi_{\lambda}^{\prime \prime}\left(  x\right)  >\varphi
_{\lambda}^{\prime \prime}\left(  x_{1}^{\lambda}\right)  =0$ for $x\in \left(
-1,x_{1}^{\lambda}\right)  $.

For $x\in \left(  x_{2}^{\lambda},+\infty \right)  $, we assert that
$\varphi_{\lambda}^{\prime \prime}\left(  x\right)  >0$. As for the converse
assertion, noting that $\varphi_{\lambda}^{\prime \prime}\left(  x_{2}%
^{\lambda}\right)  =0$ and $\lim_{x\rightarrow+\infty}\varphi_{\lambda
}^{\prime \prime}\left(  x\right)  =0$, we suppose that there exists a $\bar
{x}\in \left(  x_{2}^{\lambda},+\infty \right)  $ such that $\varphi_{\lambda
}^{\prime \prime \prime}\left(  \bar{x}\right)  =0$ and $\varphi_{\lambda
}^{\prime \prime}\left(  \bar{x}\right)  \leq0$. If $\bar{x}>\sqrt{1+2\lambda}%
$, by (\ref{8}), we deduce that
\[
\varphi_{\lambda}^{\prime \prime}\left(  \bar{x}\right)  =\frac{2\lambda \left(
2\lambda+1\right)  }{\bar{x}^{2}-2\lambda-1}\varphi_{\lambda}\left(  \bar
{x}\right)  >0\text{,}%
\]
which contradicts the assumption that $\varphi_{\lambda}^{\prime \prime}\left(
\bar{x}\right)  \leq0$. If $\bar{x}\leq \sqrt{1+2\lambda}$, we define%
\[
\hat{x}:=\min \left \{  x\in \left[  x_{2}^{\lambda},\bar{x}\right]
:\varphi_{\lambda}^{\prime \prime \prime}\left(  x\right)  =0\right \}  \text{.}%
\]
By (\ref{8}) and the fact that $\varphi_{\lambda}^{\prime \prime}\left(
x_{2}^{\lambda}\right)  =0$, we have%
\[
\varphi_{\lambda}^{\prime \prime \prime}\left(  x_{2}^{\lambda}\right)
=\frac{2\lambda \left(  2\lambda+1\right)  }{x_{2}^{\lambda}}\varphi_{\lambda
}\left(  x_{2}^{\lambda}\right)  >0\text{.}%
\]
Thus $\hat{x}>x_{2}^{\lambda}$ and $\varphi_{\lambda}^{\prime \prime \prime
}\left(  x\right)  >0$ for $x\in \left(  x_{2}^{\lambda},\hat{x}\right)  $,
which implies that $\varphi_{\lambda}^{\prime \prime}$ is strictly increasing
on $\left(  x_{2}^{\lambda},\hat{x}\right]  $. It follows that $\varphi
_{\lambda}^{\prime \prime}\left(  x\right)  >0$ for $x\in \left(  x_{2}%
^{\lambda},\hat{x}\right]  $ by $\varphi_{\lambda}^{\prime \prime}\left(
x_{2}^{\lambda}\right)  =0$, particularly $\varphi_{\lambda}^{\prime \prime
}\left(  \hat{x}\right)  >0$. Applying (\ref{8}) again, we obtain%
\[
\varphi_{\lambda}^{\prime \prime \prime}\left(  \hat{x}\right)  =\frac
{2\lambda \left(  2\lambda+1\right)  }{\hat{x}}\varphi_{\lambda}\left(  \hat
{x}\right)  -\frac{\hat{x}^{2}-2\lambda-1}{\hat{x}}\varphi_{\lambda}%
^{\prime \prime}\left(  \hat{x}\right)  >0\text{,}%
\]
which contradicts the definition of $\hat{x}$. Thus we conclude that
$\varphi_{\lambda}^{\prime \prime}\left(  x\right)  >0$ for $x\in \left(
x_{2}^{\lambda},+\infty \right)  $.

\textbf{Step 3:} Finally we show that $\varphi_{\lambda}^{\prime \prime}\left(
x\right)  >0$ for $x\in \left(  -\infty,-1\right)  $ and $\varphi_{\lambda
}^{\prime \prime}\left(  x\right)  <0$ for $x\in \left(  0,1\right)  $.

For $x\in \left(  -\infty,-1\right)  $, note that $\left(  y-x\right)
^{2}-1>0$ for $y\geq0$, then it is obvious that $\varphi_{\lambda}%
^{\prime \prime}\left(  x\right)  >0$ by (\ref{2402}).

Consider $\varphi_{\lambda}^{\prime \prime}\left(  x\right)  $ for $x\in \left(
0,1\right)  $. Let $x^{\ast}\in \left[  0,1\right]  $ denote the point such
that
\begin{equation}
\varphi_{\lambda}^{\prime \prime}\left(  x^{\ast}\right)  =\max_{x\in \left[
0,1\right]  }\varphi_{\lambda}^{\prime \prime}\left(  x\right)  \text{.}%
\label{0604}%
\end{equation}
If $x^{\ast}=0$ or $x^{\ast}=1$, then $\varphi_{\lambda}^{\prime \prime}\left(
x\right)  <0$ because we have shown that $\varphi_{\lambda}^{\prime \prime
}\left(  0\right)  <0$ and $\varphi_{\lambda}^{\prime \prime}\left(  1\right)
<0$ in Step 1. If $x^{\ast}\in \left(  0,1\right)  $, then $\varphi_{\lambda
}^{\prime \prime \prime}\left(  x^{\ast}\right)  =0$. Applying (\ref{9}) and
(\ref{10}), we have%
\[
\left[  \left(  x^{\ast}\right)  ^{2}-2\lambda-1\right]  \varphi_{\lambda
}^{\prime}\left(  x^{\ast}\right)  +2\lambda x^{\ast}\varphi_{\lambda}\left(
x^{\ast}\right)  =0
\]
and
\[
\left(  -2\lambda-1\right)  \varphi_{\lambda}^{\prime}\left(  x^{\ast}\right)
-x^{\ast}\varphi_{\lambda}^{\prime \prime}\left(  x^{\ast}\right)  =0\text{.}%
\]
It follows that%
\[
\varphi_{\lambda}^{\prime \prime}\left(  x^{\ast}\right)  =\frac{-2\lambda
-1}{x^{\ast}}\cdot \frac{-2\lambda x^{\ast}\varphi_{\lambda}\left(  x^{\ast
}\right)  }{\left(  x^{\ast}\right)  ^{2}-2\lambda-1}<0\text{.}%
\]
Thus we have $\varphi_{\lambda}^{\prime \prime}\left(  x\right)  <0$ for every
$x\in \left(  0,1\right)  $ by (\ref{0604}).

\subsection{\textbf{Proof of Lemma \ref{2107}\label{0802}}}

Let $\lambda \in \left(  0,\frac{1}{2}\right)  $ be fixed. Since
\[
\Psi_{\lambda}\left(  x\right)  =\mu_{2}\cdot \left[  \frac{\mu_{1}}{\mu_{2}%
}\varphi_{\lambda}\left(  x\right)  +\varphi_{\lambda}\left(  -x\right)
\right]  \text{,}%
\]
where $\mu_{1},\mu_{2}>0$, without loss of generality we only consider the
following type of solutions to ODE (\ref{55}):
\begin{equation}
\Psi_{\lambda,k}\left(  x\right)  =k\varphi_{\lambda}\left(  x\right)
+\varphi_{\lambda}\left(  -x\right)  , \label{12}%
\end{equation}
where $k\in \left(  0,+\infty \right)  $ is any fixed constant. It is clear that
$\Psi_{\lambda,k}$ is positive on $%
\mathbb{R}
$. The proof will be divided into two steps.

\textbf{Step 1:} For $x\in \left[  0,+\infty \right)  $, by Lemma \ref{2106}, we
can easily check that%
\begin{equation}%
\begin{array}
[c]{l}%
\Psi_{\lambda,k}^{\prime \prime}\left(  x\right)  <0\text{ for }x\in \left[
0,-x_{1}^{\lambda}\right]  \text{,}\\
\Psi_{\lambda,k}^{\prime \prime}\left(  x\right)  >0\text{ for }x\in \left[
x_{2}^{\lambda},+\infty \right)  \text{.}%
\end{array}
\label{2501}%
\end{equation}
Then there exists a $z^{\ast}\in \left(  -x_{1}^{\lambda},x_{2}^{\lambda
}\right)  $ such that $\Psi_{\lambda,k}^{\prime \prime}\left(  z^{\ast}\right)
=0$ by the continuity of $\Psi_{\lambda,k}^{\prime \prime}$. Define%
\begin{equation}
z_{2}^{\lambda,k}:=\min \left \{  x\in \left(  -x_{1}^{\lambda},x_{2}^{\lambda
}\right)  :\Psi_{\lambda,k}^{\prime \prime}\left(  x\right)  =0\right \}
\text{.} \label{13}%
\end{equation}
By (\ref{2501}) and (\ref{13}), it is obvious that
\begin{equation}
\Psi_{\lambda,k}^{\prime \prime}\left(  x\right)  <0\text{ for }x\in \left(
-x_{1}^{\lambda},z_{2}^{\lambda,k}\right)  \text{.} \label{2803}%
\end{equation}
We proceed to show that%
\begin{equation}
\Psi_{\lambda,k}^{\prime \prime}\left(  x\right)  >0\text{ for }x\in \left(
z_{2}^{\lambda,k},x_{2}^{\lambda}\right)  \text{,} \label{2804}%
\end{equation}
which implies the uniqueness of $z_{2}^{\lambda,k}$ on $\left(  -x_{1}%
^{\lambda},x_{2}^{\lambda}\right)  $.

We assert that $\Psi_{\lambda,k}^{\prime \prime}\left(  x\right)  >0$ for
$x\in \left(  z_{2}^{\lambda,k},x_{2}^{\lambda}\right)  $. Otherwise, there exists a $x^{\prime}\in \left(  z_{2}%
^{\lambda,k},x_{2}^{\lambda}\right)  $ such that $\Psi_{\lambda,k}%
^{\prime \prime \prime}\left(  x^{\prime}\right)  =0$ and $\Psi_{\lambda
,k}^{\prime \prime}\left(  x^{\prime}\right)  \leq0$ because we have $\Psi_{\lambda,k}^{\prime \prime}\left(  z_{2}^{\lambda,k}\right)  =0$ by
(\ref{13}) and $\Psi_{\lambda,k}^{\prime \prime}\left(  x_{2}^{\lambda}\right)
>0$ by (\ref{2501}). If $x^{\prime}%
>\sqrt{1+2\lambda}$, then by (\ref{8}) we obtain
\[
\Psi_{\lambda,k}^{\prime \prime \prime}\left(  x^{\prime}\right)  =\frac
{2\lambda \left(  2\lambda+1\right)  }{x^{\prime}}\Psi_{\lambda,k}\left(
x^{\prime}\right)  -\frac{\left(  x^{\prime}\right)  ^{2}-2\lambda
-1}{x^{\prime}}\Psi_{\lambda,k}^{\prime \prime}\left(  x^{\prime}\right)
>0\text{,}%
\]
which is contrary to $\Psi_{\lambda,k}^{\prime \prime \prime}\left(  x^{\prime
}\right)  =0$. If $x^{\prime}\leq \sqrt{1+2\lambda}$, we define%
\[
\hat{z}:=\min \left \{  x\in \left[  z_{2}^{\lambda,k},x^{\prime}\right]
:\Psi_{\lambda,k}^{\prime \prime \prime}\left(  x\right)  =0\right \}  \text{.}%
\]
By (\ref{8}) and the fact that $\Psi_{\lambda,k}^{\prime \prime}\left(
z_{2}^{\lambda,k}\right)  =0$, we have%
\[
\Psi_{\lambda,k}^{\prime \prime \prime}\left(  z_{2}^{\lambda,k}\right)
=\frac{2\lambda \left(  2\lambda+1\right)  }{z_{2}^{\lambda,k}}\Psi_{\lambda
,k}\left(  z_{2}^{\lambda,k}\right)  >0\text{.}%
\]
Thus $\hat{z}>z_{2}^{\lambda,k}$ and $\Psi_{\lambda,k}^{\prime \prime \prime
}\left(  x\right)  >0$ for $x\in \left(  z_{2}^{\lambda,k},\hat{z}\right)  $,
which implies that $\Psi_{\lambda,k}^{\prime \prime}$ is strictly increasing on
$\left(  z_{2}^{\lambda,k},\hat{z}\right]  $. It follows that $\Psi
_{\lambda,k}^{\prime \prime}\left(  x\right)  >0$ for $x\in \left(
z_{2}^{\lambda,k},\hat{z}\right]  $ by $\Psi_{\lambda,k}^{\prime \prime}\left(
z_{2}^{\lambda,k}\right)  =0$, particularly $\Psi_{\lambda,k}^{\prime \prime
}\left(  \hat{z}\right)  >0$. Applying (\ref{8}) again, we can easily verify
that%
\[
\Psi_{\lambda,k}^{\prime \prime \prime}\left(  \hat{z}\right)  =\frac
{2\lambda \left(  2\lambda+1\right)  }{\hat{z}}\Psi_{\lambda,k}\left(  \hat
{z}\right)  -\frac{\hat{z}^{2}-2\lambda-1}{\hat{z}}\Psi_{\lambda,k}%
^{\prime \prime}\left(  \hat{z}\right)  >0\text{,}%
\]
which contradicts the definition of $\hat{z}$. Consequently, (\ref{2804}) is proved.

\textbf{Step 2:} For $x\in \left(  -\infty,0\right)  $, note that
\[
\Psi_{\lambda,k}^{\prime \prime}\left(  x\right)  =k\cdot \left[  \varphi
_{\lambda}^{\prime \prime}\left(  x\right)  +\frac{1}{k}\varphi_{\lambda
}^{\prime \prime}\left(  -x\right)  \right]  =k\cdot \Psi_{\lambda,\frac{1}{k}%
}^{\prime \prime}\left(  -x\right)  \text{,}%
\]
where $k\in \left(  0,+\infty \right)  $. Therefore, by Step 1, $z_{1}%
^{\lambda,k}=-z_{2}^{\lambda,\frac{1}{k}}\in \left(  -x_{2}^{\lambda}%
,x_{1}^{\lambda}\right)  $ satisfies $\Psi_{\lambda,k}^{\prime \prime}\left(
z_{1}^{\lambda,k}\right)  =0$ and the uniqueness of $z_{1}^{\lambda,k}$ is
obvious, and we also have%
\[
\left \{
\begin{array}
[c]{l}%
\Psi_{\lambda,k}^{\prime \prime}\left(  x\right)  <0\text{ for }x\in \left(
z_{1}^{\lambda,k},0\right)  \text{,}\\
\Psi_{\lambda,k}^{\prime \prime}\left(  x\right)  >0\text{ for }x\in \left(
-\infty,z_{1}^{\lambda,k}\right)  \text{.}%
\end{array}
\right.
\]

\subsection{\textbf{Proof of Lemma \ref{2303}\label{1501}}}

It is obvious that $x_{1}^{\lambda}$ is continuous in $\lambda \in \left(
0,\frac{1}{2}\right)  $.

Let $x\in \left(  -1,0\right)  $ be any fixed constant. For $\lambda
\in \left(  0,\frac{1}{2}\right)  $, define
\begin{equation}
f_{x}\left(  \lambda \right)  :=\varphi_{\lambda}^{\prime \prime}\left(
x\right)  =\int_{0}^{+\infty}y^{-2\lambda}\left[  \left(  y-x\right)
^{2}-1\right]  \exp \left \{  -\frac{\left(  y-x\right)  ^{2}}{2}\right \}
dy\text{,}\label{0805}%
\end{equation}
We only need to show that there exists a unique $\tilde{\lambda}_{1}\in \left(
0,\frac{1}{2}\right)  $ such that
\begin{equation}
\left \{
\begin{array}
[c]{l}%
f_{x}\left(  \tilde{\lambda}_{1}\right)  =0\text{,}\\
f_{x}\left(  \lambda \right)  >0\text{ for }\lambda \in \left(  0,\tilde{\lambda
}_{1}\right)  \text{,}\\
f_{x}\left(  \lambda \right)  <0\text{ for }\lambda \in \left(  \tilde{\lambda
}_{1},\frac{1}{2}\right)  \text{.}%
\end{array}
\right.  \label{1803}%
\end{equation}
Indeed, for $0<\lambda^{\prime}<\lambda^{\prime \prime}<\frac{1}{2}$, since
$f_{x_{1}^{\lambda^{\prime \prime}}}\left(  \lambda^{\prime \prime}\right)
=\varphi_{\lambda^{\prime \prime}}^{\prime \prime}\left(  x_{1}^{\lambda
^{\prime \prime}}\right)  =0$ by Notation \ref{3.2}, we have $\varphi
_{\lambda^{\prime}}^{\prime \prime}\left(  x_{1}^{\lambda^{\prime \prime}%
}\right)  =f_{x_{1}^{\lambda^{\prime \prime}}}\left(  \lambda^{\prime}\right)
>0$ by (\ref{1803}). Hence $x_{1}^{\lambda^{\prime \prime}}<x_{1}%
^{\lambda^{\prime}}$ by Lemma \ref{2106}, which implies that $x_{1}^{\lambda}$
is strictly decreasing in $\lambda \in \left(  0,\frac{1}{2}\right)  $. Note
that $x\in \left(  -1,0\right)  $ is arbitrary, then we get (\ref{3001}). The
proof of (\ref{1803}) is divided into two steps.

\textbf{Step 1:} We first show that $\lim_{\lambda \rightarrow0}f_{x}\left(
\lambda \right)  >0$ and $\lim_{\lambda \rightarrow \frac{1}{2}}f_{x}\left(
\lambda \right)  =-\infty$.

By dominated convergence theorem, it is easily seen
that%
\begin{equation}
\lim_{\lambda \rightarrow0}f_{x}\left(  \lambda \right)  =\int_{0}^{+\infty
}\left[  \left(  y-x\right)  ^{2}-1\right]  \exp \left \{  -\frac{\left(
y-x\right)  ^{2}}{2}\right \}  dy=-x\cdot \exp \left \{  -\frac{x^{2}}{2}\right \}
>0\text{.} \label{0204}%
\end{equation}

As for the case when $\lambda \rightarrow \frac{1}{2}$, by substitution $t=y-x$
in (\ref{0805}), we have%
\begin{equation}%
\begin{array}
[c]{ll}
& \lim_{\lambda \rightarrow \frac{1}{2}}f_{x}\left(  \lambda \right) \\
= & \lim_{\lambda \rightarrow \frac{1}{2}}\int_{-x}^{+\infty}\left(  t+x\right)
^{-2\lambda}\left(  t^{2}-1\right)  \exp \left \{  -\frac{t^{2}}{2}\right \}
dt\\
= & \lim_{\lambda \rightarrow \frac{1}{2}}\int_{-x}^{1}\left(  t+x\right)
^{-2\lambda}\left(  t^{2}-1\right)  \exp \left \{  -\frac{t^{2}}{2}\right \}
dt+\lim_{\lambda \rightarrow \frac{1}{2}}\int_{1}^{+\infty}\left(  t+x\right)
^{-2\lambda}\left(  t^{2}-1\right)  \exp \left \{  -\frac{t^{2}}{2}\right \}
dt\text{.}%
\end{array}
\label{0201}%
\end{equation}
For the first term of (\ref{0201}), by monotone convergence theorem, we can
easily check that%
\begin{equation}
\lim_{\lambda \rightarrow \frac{1}{2}}\int_{-x}^{1}\left(  t+x\right)
^{-2\lambda}\left(  t^{2}-1\right)  \exp \left \{  -\frac{t^{2}}{2}\right \}
dt=\int_{-x}^{1}\left(  t+x\right)  ^{-1}\left(  t^{2}-1\right)  \exp \left \{
-\frac{t^{2}}{2}\right \}  dt=-\infty \text{.} \label{1804}%
\end{equation}
And for the second term of (\ref{0201}), by dominated convergence theorem, we
have
\begin{equation}
\lim_{\lambda \rightarrow \frac{1}{2}}\int_{1}^{+\infty}\left(  t+x\right)
^{-2\lambda}\left(  t^{2}-1\right)  \exp \left \{  -\frac{t^{2}}{2}\right \}
dt=\int_{1}^{+\infty}\left(  t+x\right)  ^{-1}\left(  t^{2}-1\right)
\exp \left \{  -\frac{t^{2}}{2}\right \}  dt<+\infty \label{1805}%
\end{equation}
Substituting (\ref{1804}) and (\ref{1805}) into (\ref{0201}), we conclude that
$\lim_{\lambda \rightarrow \frac{1}{2}}f_{x}\left(  \lambda \right)  =-\infty$.

\textbf{Step 2:} By Step 1 and the continuity of $f_{x}\left(  \lambda \right)
$ in $\lambda \in \left(  0,\frac{1}{2}\right)  $, we know that there exists
$\tilde{\lambda}_{1}\in \left(  0,\frac{1}{2}\right)  $ such that $f_{x}\left(
\tilde{\lambda}_{1}\right)  =0$. Now we show that $f_{x}\left(  \lambda
\right)  >0$ for $\lambda \in \left(  0,\tilde{\lambda}_{1}\right)  $ and
$f_{x}\left(  \lambda \right)  <0$ for $\lambda \in \left(  \tilde{\lambda}%
_{1},\frac{1}{2}\right)  $, which directly implies the uniqueness of
$\tilde{\lambda}_{1}$ on $\left(  0,\frac{1}{2}\right)  $.

Set%
\begin{equation}
\bar{\lambda}_{1}=\max \left \{  \lambda \in \left(  0,\frac{1}{2}\right)
:f_{x}\left(  \lambda \right)  =0\right \}  \text{.}\label{0807}%
\end{equation}
Obviously, we have $f_{x}\left(  \lambda \right)  <0$ for $\lambda \in \left(
\bar{\lambda}_{1},\frac{1}{2}\right)  $. We assert that $f_{x}\left(
\lambda \right)  >0$ for $\lambda \in \left(  0,\bar{\lambda}_{1}\right)  $,
which is the desired conclusion. Suppose the assertion is false. Then, by
(\ref{0204}) and the fact that $f_{x}\left(  \bar{\lambda}_{1}\right)  =0$,
there exists a $\lambda^{\ast}\in \left(  0,\bar{\lambda}_{1}\right)  $ such
that%
\[
f_{x}^{\prime}\left(  \lambda^{\ast}\right)  =0\text{ and }f_{x}\left(
\lambda^{\ast}\right)  \leq0\text{.}%
\]
Substituting $t=y-x$ into (\ref{0805}), we get that%
\begin{equation}
\int_{-x}^{1-x}\left(  t+x\right)  ^{-2\lambda^{\ast}}\left(  t^{2}-1\right)
\exp \left \{  -\frac{t^{2}}{2}\right \}  dt+\int_{1-x}^{+\infty}\left(
t+x\right)  ^{-2\lambda^{\ast}}\left(  t^{2}-1\right)  \exp \left \{
-\frac{t^{2}}{2}\right \}  dt\leq0\text{.}\label{0901}%
\end{equation}
Obviously, we have $\int_{1-x}^{+\infty}\left(  t+x\right)  ^{-2\lambda^{\ast
}}\left(  t^{2}-1\right)  \exp \left \{  -\frac{t^{2}}{2}\right \}  dt>0$. Then
by (\ref{0901}) we get that%
\[
\int_{-x}^{1-x}\left(  t+x\right)  ^{-2\lambda^{\ast}}\left(  t^{2}-1\right)
\exp \left \{  -\frac{t^{2}}{2}\right \}  dt<0\text{,}%
\]
which implies%
\begin{equation}
0<\int_{1}^{1-x}\left(  t+x\right)  ^{-2\lambda^{\ast}}\left(  t^{2}-1\right)
\exp \left \{  -\frac{t^{2}}{2}\right \}  dt<-\int_{-x}^{1}\left(  t+x\right)
^{-2\lambda^{\ast}}\left(  t^{2}-1\right)  \exp \left \{  -\frac{t^{2}}%
{2}\right \}  dt\text{.}\label{1901}%
\end{equation}
On the other hand, by\ substitution $t=y-x$, we can easily compute that%
\begin{equation}
f_{x}^{\prime}\left(  \lambda^{\ast}\right)  =-2\int_{-x}^{+\infty}\ln \left(
t+x\right)  \left(  t+x\right)  ^{-2\lambda^{\ast}}\left(  t^{2}-1\right)
\exp \left \{  -\frac{t^{2}}{2}\right \}  dt=0\text{.}\label{0206}%
\end{equation}
Noting that $\ln \left(  t+x\right)  <0$ for $t\in \left(  1,1-x\right)  $, from
(\ref{1901}) we can check that%
\[%
\begin{array}
[c]{ll}
& \int_{1}^{1-x}\ln \left(  t+x\right)  \left(  t+x\right)  ^{-2\lambda^{\ast}%
}\left(  t^{2}-1\right)  \exp \left \{  -\frac{t^{2}}{2}\right \}  dt\\
> & \ln \left(  1+x\right)  \cdot \int_{1}^{1-x}\left(  t+x\right)
^{-2\lambda^{\ast}}\left(  t^{2}-1\right)  \exp \left \{  -\frac{t^{2}}%
{2}\right \}  dt\\
> & \ln \left(  1+x\right)  \cdot \left(  -\int_{-x}^{1}\left(  t+x\right)
^{-2\lambda^{\ast}}\left(  t^{2}-1\right)  \exp \left \{  -\frac{t^{2}}%
{2}\right \}  dt\right)  \\
> & -\int_{-x}^{1}\ln \left(  t+x\right)  \left(  t+x\right)  ^{-2\lambda
^{\ast}}\left(  t^{2}-1\right)  \exp \left \{  -\frac{t^{2}}{2}\right \}
dt\text{,}%
\end{array}
\]
which implies%
\begin{equation}
\int_{-x}^{1-x}\ln \left(  t+x\right)  \left(  t+x\right)  ^{-2\lambda^{\ast}%
}\left(  t^{2}-1\right)  \exp \left \{  -\frac{t^{2}}{2}\right \}  dt>0\text{.}%
\label{0205}%
\end{equation}
Obviously, we have%
\begin{equation}
\int_{1-x}^{+\infty}\ln \left(  t+x\right)  \left(  t+x\right)  ^{-2\lambda
^{\ast}}\left(  t^{2}-1\right)  \exp \left \{  -\frac{t^{2}}{2}\right \}
dt>0\text{.}\label{16}%
\end{equation}
The combination of (\ref{0205}) and (\ref{16}) is contrary to (\ref{0206}).

\subsection{\textbf{Proof of Lemma \ref{1801}\label{0803}}}

Before giving the proof of Lemma \ref{1801}, we need the\ following lemmas.

\begin{lemma}
\label{61405}Let $x_{2}^{\lambda}$ be defined by Notation \ref{3.2}. Then
$x_{2}^{\lambda}$ is continuous and\ strictly\ decreasing in $\lambda
\in \left(  0,\frac{1}{2}\right)  $, and
\[
\lim_{\lambda \rightarrow0}x_{2}^{\lambda}=+\infty,\lim_{\lambda \rightarrow
\frac{1}{2}}x_{2}^{\lambda}=1.
\]

\end{lemma}

\begin{proof}
The continuity of $x_{2}^{\lambda}$ in $\lambda \in \left(  0,\frac{1}%
{2}\right)  $ is obvious.

Let $x\in \left(  1,+\infty \right)  $ be fixed.
Similarly to the proof of Lemma \ref{2303}, by Lemma \ref{2106}, we only need
to show that there exists a unique $\tilde{\lambda}_{2}\in \left(  0,\frac
{1}{2}\right)  $ such that%
\begin{equation}
\left \{
\begin{array}
[c]{l}%
f_{x}\left(  \tilde{\lambda}_{2}\right)  =0,\\
f_{x}\left(  \lambda \right)  <0\text{ for }\lambda \in \left(  0,\tilde{\lambda
}_{2}\right)  ,\\
f_{x}\left(  \lambda \right)  >0\text{ for }\lambda \in \left(  \tilde{\lambda
}_{2},\frac{1}{2}\right)  ,
\end{array}
\right.  \label{62802}%
\end{equation}
where $f_{x}$ is defined by (\ref{0805}) for any fixed $x\in \left(  1,+\infty \right)  $. The proof of (\ref{62802}) is
divided into two steps.

\textbf{Step 1:} First we prove that $\lim_{\lambda \rightarrow0}f_{x}\left(
\lambda \right)  <0$ and $\lim_{\lambda \rightarrow \frac{1}{2}}f_{x}\left(
\lambda \right)  =+\infty$.

By dominated convergence theorem, it is obvious
that
\begin{equation}
\lim_{\lambda \rightarrow0}f_{x}\left(  \lambda \right)  =\int_{0}^{+\infty
}\left[  \left(  y-x\right)  ^{2}-1\right]  \exp \left \{  -\frac{\left(
y-x\right)  ^{2}}{2}\right \}  dy=-x\cdot \exp \left \{  -\frac{x^{2}}{2}\right \}
<0. \label{62803}%
\end{equation}
As for the case when $\lambda \rightarrow \frac{1}{2}$, by substitution $t=y-x$,
we have%
\begin{equation}%
\begin{array}
[c]{rl}
& \lim_{\lambda \rightarrow \frac{1}{2}}f_{x}\left(  \lambda \right)  \\
= & \lim_{\lambda \rightarrow \frac{1}{2}}\int_{0}^{+\infty}y^{-2\lambda}\left[
\left(  y-x\right)  ^{2}-1\right]  \exp \left \{  -\frac{\left(  y-x\right)
^{2}}{2}\right \}  dy\\
= & \lim_{\lambda \rightarrow \frac{1}{2}}\int_{-x}^{+\infty}\left(  t+x\right)
^{-2\lambda}\left(  t^{2}-1\right)  \exp \left \{  -\frac{t^{2}}{2}\right \}
dt\\
= & \lim_{\lambda \rightarrow \frac{1}{2}}\int_{-x}^{-1\wedge \left(  1-x\right)
}\left(  t+x\right)  ^{-2\lambda}\left(  t^{2}-1\right)  \exp \left \{
-\frac{t^{2}}{2}\right \}  dt\\
& +\lim_{\lambda \rightarrow \frac{1}{2}}\int_{-1\wedge \left(  1-x\right)
}^{+\infty}\left(  t+x\right)  ^{-2\lambda}\left(  t^{2}-1\right)
\exp \left \{  -\frac{t^{2}}{2}\right \}  dt,
\end{array}
\label{0209}%
\end{equation}
For the first term of (\ref{0209}), by monotone convergence theorem, we can
easily check that%
\[
\lim_{\lambda \rightarrow \frac{1}{2}}\int_{-x}^{-1\wedge \left(  1-x\right)
}\left(  t+x\right)  ^{-2\lambda}\left(  t^{2}-1\right)  \exp \left \{
-\frac{t^{2}}{2}\right \}  dt=\int_{-x}^{-1\wedge \left(  1-x\right)  }\left(
t+x\right)  ^{-1}\left(  t^{2}-1\right)  \exp \left \{  -\frac{t^{2}}%
{2}\right \}  dt=+\infty \text{.}%
\]
Noting that the second term of (\ref{0209})\ converges to a finite number, we
obtain $\lim_{\lambda \rightarrow \frac{1}{2}}f_{x}\left(  \lambda \right)
=+\infty$.

\textbf{Step 2:} We know that there exists a $\tilde{\lambda}_{2}\in \left(
0,\frac{1}{2}\right)  $ such that $f_{x}\left(  \tilde{\lambda}_{2}\right)
=0$ from Step 1 and the continuity of $f_{x}\left(  \cdot \right)  $. Now we
show the uniqueness of $\tilde{\lambda}_{2}$ on $\left(  0,\frac{1}{2}\right)
$, which will be divided into four parts.

\textbf{(i).} First we prove that $\lim_{\lambda \rightarrow \frac{1}{2}}%
x_{2}^{\lambda}=1$ and $\lim_{\lambda \rightarrow0}x_{2}^{\lambda}=+\infty$.
For every $\varepsilon>0$, since $\lim_{\lambda \rightarrow \frac{1}{2}%
}f_{1+\varepsilon}\left(  \lambda \right)  =+\infty$, there exists a
$\lambda_{0}\in \left(  0,\frac{1}{2}\right)  $ such that $f_{1+\varepsilon
}\left(  \lambda \right)  >0$ for $\lambda \in \left(  \lambda_{0},\frac{1}%
{2}\right)  $. Thus $x_{2}^{\lambda}<1+\varepsilon$ by Lemma \ref{2106}. Then
we have $\lim_{\lambda \rightarrow \frac{1}{2}}x_{2}^{\lambda}=1$. The proof of
$\lim_{\lambda \rightarrow0}x_{2}^{\lambda}=+\infty$ is similar.

\textbf{(ii).} Let
\begin{align*}
\lambda_{1}\left(  x\right)   &  =\max \left \{  \lambda \in \left(  0,\frac{1}%
{2}\right)  :f_{x}\left(  \lambda \right)  =0\right \}  ,\\
\lambda_{2}\left(  x\right)   &  =\min \left \{  \lambda \in \left(  0,\frac{1}%
{2}\right)  :f_{x}\left(  \lambda \right)  =0\right \}  .
\end{align*}
Now we prove that $\lambda_{1}\left(  x\right)  $ and $\lambda_{2}\left(
x\right)  $ are strictly decreasing and continuous in $x\in \left(
1,+\infty \right)  $.

For $x^{\prime}\in \left(  x,+\infty \right)  $, noting
that $f_{x}\left[  \lambda_{2}\left(  x\right)  \right]  =0$, we have
$f_{x^{\prime}}\left[  \lambda_{2}\left(  x\right)  \right]  >0$ by Lemma
\ref{2106}. According to Step 1 and the definition of $\lambda_{2}\left(
x^{\prime}\right)  $, we get that $f_{x^{\prime}}\left(  \lambda \right)
\leq0$ for very $\lambda \in \left(  0,\lambda_{2}\left(  x^{\prime}\right)
\right]  $. Thus $\lambda_{2}\left(  x^{\prime}\right)  <\lambda_{2}\left(
x\right)  $. It follows that $\lambda_{2}\left(  x\right)  $ is strictly
decreasing on $\left(  1,+\infty \right)  $.

As for continuity, it is obvious that $\lambda_{2}\left(  x\right)  $ is
right-continuous in $x\in \left(  1,+\infty \right)  $. Now we show the
left-continuity. Noting that, for $\lambda \in \left(  0,\frac{1}{2}\right)  $,
\begin{equation}%
\begin{array}
[c]{ll}
& \int_{0}^{+\infty}\left(  \ln \frac{y}{x-1}\ln \frac{y}{x+1}\right)
y^{-2\lambda}\left[  \left(  y-x\right)  ^{2}-1\right]  \exp \left \{
-\frac{\left(  y-x\right)  ^{2}}{2}\right \}  dy\\
= & \frac{1}{4}f_{x}^{\prime \prime}\left(  \lambda \right)  +\frac{1}{2}%
\ln \left(  x^{2}-1\right)  f_{x}^{\prime}\left(  \lambda \right)  +\ln \left(
x-1\right)  \ln \left(  x+1\right)  f_{x}\left(  \lambda \right)  \\
> & 0,
\end{array}
\label{3002}%
\end{equation}
we assert that, for every $\varepsilon>0$, there exists a $\mu \in \left(
\lambda_{2}\left(  x\right)  ,\lambda_{2}\left(  x\right)  +\varepsilon
\right)  $ such that $f_{x}\left(  \mu \right)  >0$. Otherwise, $\lambda
_{2}\left(  x\right)  $ is the local maximum point of $f_{x}\left(
\cdot \right)  $. Then $f_{x}\left[  \lambda_{2}\left(  x\right)  \right]  =0$,
$f_{x}^{\prime}\left[  \lambda_{2}\left(  x\right)  \right]  =0$ and
$f_{x}^{\prime \prime}\left[  \lambda_{2}\left(  x\right)  \right]  \leq0$,
which contradicts (\ref{3002}).\ According to Lemma \ref{2106}, there exists
$\tilde{x}^{\prime}\in \left(  1,x\right)  $ such that $f_{\tilde{x}^{\prime}%
}\left(  \mu \right)  =0$. It follows that $\lambda_{2}\left(  \tilde
{x}^{\prime}\right)  \leq \mu$ by the definition of $\lambda_{2}\left(
x\right)  $. Since $\lambda_{2}\left(  x\right)  $ is decreasing in\textbf{
}$x\in \left(  1,+\infty \right)  $, we have $\lambda_{2}\left(
x\right)  <\lambda_{2}\left(  \tilde{x}^{\prime}\right)  \leq \mu$, which
implies the left-continuity of $\lambda_{2}\left(  x\right)  $ in $x\in \left(
1,+\infty \right)  $. Consequently, $\lambda_{2}\left(  x\right)  $ is
continuous in $x\in \left(  1,+\infty \right)  $.

The proof for $\lambda
_{1}\left(  x\right)  $ is similar.

\textbf{(iii).} We next show that
\[
\lim_{x\rightarrow1}\lambda_{i}\left(  x\right)  =\frac{1}{2},\text{ }%
\lim_{x\rightarrow+\infty}\lambda_{i}\left(  x\right)  =0,\text{ }i=1,2.
\]
Since $\lim_{\lambda \rightarrow \frac{1}{2}}x_{2}^{\lambda}=1$, it is easily
seen that $\lim_{x\rightarrow1}\lambda_{1}\left(  x\right)  =\frac{1}{2}$. Set
$\lim_{x\rightarrow+\infty}\lambda_{1}\left(  x\right)  :=\lambda_{1}^{\ast
}\geq0$. Since $f_{x}\left[  \lambda_{1}\left(  x\right)  \right]  =0$, for
each fixed $x_{0}\in \left(  1,x\right)  $, we have $f_{x_{0}}\left[
\lambda_{1}\left(  x\right)  \right]  <0$ according to Lemma \ref{2106}.
Suppose $\lambda_{1}^{\ast}>0$, then we have $f_{x_{0}}\left(  \lambda
_{1}^{\ast}\right)  \leq0$ for $x_{0}\in \left(  1,+\infty \right)  $ by letting
$x\rightarrow+\infty$, which contradicts Lemma \ref{2106}. Therefore
$\lambda_{1}^{\ast}=0$. Similarly we can prove that $\lim_{x\rightarrow
1}\lambda_{2}\left(  x\right)  =\frac{1}{2}$ and $\lim_{x\rightarrow+\infty
}\lambda_{2}\left(  x\right)  =0$.

\textbf{(iv).} Finally, noting that $\lambda_{1}\left(  x\right)  \geq
\lambda_{2}\left(  x\right)  $ for every $x\in \left(  1,+\infty \right)  $, by (ii) and (iii), we know that for every $\lambda^{\prime}%
\in \left(  0,\frac{1}{2}\right)  $ there exist $x_{1},x_{2}\in \left(
1,+\infty \right)  $ satisfying $x_{1}\geq x_{2}$ such that $\lambda_{2}\left(
x_{2}\right)  =\lambda_{1}\left(  x_{1}\right)  =\lambda^{\prime}$. It follows
that%
\[
f_{x_{i}}\left[  \lambda_{i}\left(  x_{i}\right)  \right]  =f_{x_{i}}\left(
\lambda^{\prime}\right)  =0,\text{ }i=1,2.
\]
According to Lemma \ref{2106}, we have $x_{1}=x_{2}$, which implies
$\lambda_{1}\left(  x\right)  =\lambda_{2}\left(  x\right)  $ by the
arbitrariness of $\lambda^{\prime}\in \left(  0,\frac{1}{2}\right)  $. Then the
uniqueness of $\tilde{\lambda}_{2}\in \left(  0,\frac{1}{2}\right)  $ is proved.

Consequently, we conclude that $x_{2}^{\lambda}$ is decreasing in $\lambda
\in \left(  0,\frac{1}{2}\right)  $.
\end{proof}

\begin{lemma}
\label{2503}Let $z_{2}^{\lambda,k}$ be defined by (\ref{13}). Then for every
$\lambda \in \left(  0,\frac{1}{2}\right)  $ and $k\in \left[  1,+\infty \right)
$, $z_{2}^{\lambda,k}>1$.
\end{lemma}

\begin{proof}
By (\ref{2801}) in Subsection \ref{0801}, we know that $\varphi_{\lambda
}^{\prime \prime}\left(  1\right)  +\varphi_{\lambda}^{\prime \prime}\left(
-1\right)  <0$, and by Lemma \ref{2106} we have $\varphi_{\lambda}%
^{\prime \prime}\left(  1\right)  <0$. Thus, for $k>1$,
\[
\Psi_{\lambda,k}^{\prime \prime}\left(  1\right)  =k\varphi_{\lambda}%
^{\prime \prime}\left(  1\right)  +\varphi_{\lambda}^{\prime \prime}\left(
-1\right)  =\left(  k-1\right)  \varphi_{\lambda}^{\prime \prime}\left(
1\right)  +\left[  \varphi_{\lambda}^{\prime \prime}\left(  1\right)
+\varphi_{\lambda}^{\prime \prime}\left(  -1\right)  \right]  <0\text{.}%
\]
which yields $z_{2}^{\lambda,k}>1$ according to Lemma \ref{2107}.
\end{proof}

Now we prove Lemma \ref{1801}. The continuity of $z^{\lambda}$ in $\lambda
\in \left(  0,\frac{1}{2}\right)  $ is obvious.

For each fixed $x\in \left(  1,+\infty \right)  $, define%
\[
F_{x}\left(  \lambda \right)  :=\varphi_{\lambda}^{\prime \prime}\left(
x\right)  +\varphi_{\lambda}^{\prime \prime}\left(  -x\right)  =f_{x}\left(
\lambda \right)  +g_{x}\left(  \lambda \right)  \text{,}%
\]
where $f_{x}\left(  \lambda \right)  $ is defined by (\ref{0805}), and
\[
g_{x}\left(  \lambda \right)  :=\int_{0}^{+\infty}y^{-2\lambda}\left[  \left(
y+x\right)  ^{2}-1\right]  \exp \left \{  -\frac{\left(  y+x\right)  ^{2}}%
{2}\right \}  dy\text{.}%
\]
And for each $\delta>0$, define%
\begin{equation}
F_{x}^{\delta}\left(  \lambda \right)  :=\Psi_{\lambda,1+\delta}^{\prime \prime
}\left(  x\right)  =\left(  1+\delta \right)  f_{x}\left(  \lambda \right)
+g_{x}\left(  \lambda \right)  ,\label{1902}%
\end{equation}
where $\Psi_{\lambda,1+\delta}$ is defined by (\ref{12}). By Lemma \ref{2503},
we have $z_{2}^{\lambda,1+\delta}>1$ and $z^{\lambda}>1$, where $z_{2}%
^{\lambda,1+\delta}$ is defined by (\ref{13}) and $z^{\lambda}$ is defined by
Notation \ref{3.2}.

We first show that $z_{2}^{\lambda,1+\delta}$ is decreasing in $\lambda
\in \left(  0,\frac{1}{2}\right)  $ and $\lim_{\lambda \rightarrow \frac{1}{2}%
}z_{2}^{\lambda,1+\delta}=1$ for every fixed $\delta>0$. Similarly to the
proof of Lemma \ref{2303}, by Lemma \ref{2107}, we only need to show that
there exists a unique $\lambda_{\delta}\in \left(  0,\frac{1}{2}\right)  $ such
that
\[
\left \{
\begin{array}
[c]{l}%
F_{x}^{\delta}\left(  \lambda_{\delta}\right)  =0\text{,}\\
F_{x}^{\delta}\left(  \lambda \right)  <0\text{ for }\lambda \in \left(
0,\lambda_{\delta}\right)  \text{,}\\
F_{x}^{\delta}\left(  \lambda \right)  >0\text{ for }\lambda \in \left(
\lambda_{\delta},\frac{1}{2}\right)  \text{.}%
\end{array}
\right.
\]
The proof is divided into two steps.

\textbf{Step 1:} We prove that $\lim_{\lambda \rightarrow0}F_{x}^{\delta
}\left(  \lambda \right)  <0$ and $\lim_{\lambda \rightarrow \frac{1}{2}}%
F_{x}^{\delta}\left(  \lambda \right)  =+\infty$. By dominated convergence
theorem, it is easily seen that%
\begin{equation}%
\begin{array}
[c]{ll}
& \lim_{\lambda \rightarrow0}F_{x}^{\delta}\left(  \lambda \right) \\
= & \left(  1+\delta \right)  \int_{0}^{+\infty}\left[  \left(  y-x\right)
^{2}-1\right]  \exp \left \{  -\frac{\left(  y-x\right)  ^{2}}{2}\right \}
dy+\int_{0}^{+\infty}\left[  \left(  y+x\right)  ^{2}-1\right]  \exp \left \{
-\frac{\left(  y+x\right)  ^{2}}{2}\right \}  dy\\
= & \left(  1+\delta \right)  \allowbreak \left(  -x\exp \left \{  -\frac{x^{2}%
}{2}\right \}  \right)  +x\exp \left \{  -\frac{x^{2}}{2}\right \}  <0\text{.}%
\end{array}
\label{62805}%
\end{equation}
As for the case when $\lambda \rightarrow \frac{1}{2}$,%
\begin{equation}
\lim_{\lambda \rightarrow \frac{1}{2}}F_{x}^{\delta}\left(  \lambda \right)
=\left(  1+\delta \right)  \lim_{\lambda \rightarrow \frac{1}{2}}f_{x}\left(
\lambda \right)  +\lim_{\lambda \rightarrow \frac{1}{2}}g_{x}\left(
\lambda \right)  \text{.} \label{0208}%
\end{equation}
The second term is positive which is obvious, and the first term diverges to
$+\infty$ by the Step 1 in the proof of Lemma \ref{61405}. Thus $\lim
_{\lambda \rightarrow \frac{1}{2}}F_{x}^{\delta}\left(  \lambda \right)
=+\infty$.

\textbf{Step 2:} We know that there exists $\tilde{\lambda}_{\delta}\in \left(
0,\frac{1}{2}\right)  $ such that $F_{x}^{\delta}\left(  \tilde{\lambda
}_{\delta}\right)  =0$ by Step 1 and the continuity of the function, now we
show the uniqueness of $\tilde{\lambda}_{\delta}$ on $\left(  0,\frac{1}%
{2}\right)  $. The proof is divided into two parts.

\textbf{(i).} We first prove the uniqueness of $\tilde{\lambda}_{\delta}$ for
$x\in \left[  2,+\infty \right)  $.

By simple calculation, we have%
\[%
\begin{array}
[c]{l}%
f_{x}^{\prime}\left(  \lambda \right)  =-2\int_{0}^{+\infty}\left(  \ln
y\right)  y^{-2\lambda}\left[  \left(  y-x\right)  ^{2}-1\right]  \exp \left \{
-\frac{\left(  y-x\right)  ^{2}}{2}\right \}  dy,\\
f_{x}^{\prime \prime}\left(  \lambda \right)  =4\int_{0}^{+\infty}\left(  \ln
y\right)  ^{2}y^{-2\lambda}\left[  \left(  y-x\right)  ^{2}-1\right]
\exp \left \{  -\frac{\left(  y-x\right)  ^{2}}{2}\right \}  dy,
\end{array}
\]
and%
\[%
\begin{array}
[c]{l}%
g_{x}^{\prime}\left(  \lambda \right)  =-2\int_{0}^{+\infty}\left(  \ln
y\right)  y^{-2\lambda}\left[  \left(  y+x\right)  ^{2}-1\right]  \exp \left \{
-\frac{\left(  y+x\right)  ^{2}}{2}\right \}  dy,\\
g_{x}^{\prime \prime}\left(  \lambda \right)  =4\int_{0}^{+\infty}\left(  \ln
y\right)  ^{2}y^{-2\lambda}\left[  \left(  y+x\right)  ^{2}-1\right]
\exp \left \{  -\frac{\left(  y+x\right)  ^{2}}{2}\right \}  dy.
\end{array}
\]

For $g\left( \lambda \right)$, it is obvious that $g_{x}^{\prime \prime}\left(  \lambda \right)  >0$ for every
$\lambda \in \left(  0,\frac{1}{2}\right)  $. Now we show that $g_{x}^{\prime
}\left(  0\right)  >0$. It is clear that%
\begin{equation}%
\begin{array}
[c]{cl}
& \int_{0}^{1}\left(  \ln y\right)  \left[  \left(  y+x\right)  ^{2}-1\right]
\exp \left \{  -\frac{\left(  y+x\right)  ^{2}}{2}\right \}  dy\\
\leq & \left[  \left(  x+1\right)  ^{2}-1\right]  \exp \left \{  -\frac{\left(
x+1\right)  ^{2}}{2}\right \}  \int_{0}^{1}\ln ydy=-\left[  \left(  x+1\right)
^{2}-1\right]  \exp \left \{  -\frac{\left(  x+1\right)  ^{2}}{2}\right \}  .
\end{array}
\label{61201}%
\end{equation}
Noting that $\max_{y\in \left(  1,+\infty \right)  }\left[  \left(  \ln
y\right)  \exp \left \{  -\frac{\left(  y+x\right)  ^{2}}{4}\right \}  \right]
=\left(  \ln y^{\ast}\right)  \exp \left \{  -\frac{\left(  y^{\ast}+x\right)
^{2}}{4}\right \}  $, where $y^{\ast}\in \left(  1,+\infty \right)  $ satisfies
\begin{equation}
\left[  \frac{1}{y^{\ast}}-\frac{y^{\ast}+x}{2}\left(  \ln y^{\ast}\right)
\right]  \exp \left \{  -\frac{\left(  y^{\ast}+x\right)  ^{2}}{4}\right \}  =0,
\end{equation}
we can easily check that%
\[%
\begin{array}
[c]{rl}
& \int_{1}^{+\infty}\left(  \ln y\right)  \left[  \left(  y+x\right)
^{2}-1\right]  \exp \left \{  -\frac{\left(  y+x\right)  ^{2}}{2}\right \}  dy\\
= & \int_{1}^{+\infty}\left(  \ln y\right)  \exp \left \{  -\frac{\left(
y+x\right)  ^{2}}{4}\right \}  \cdot \left[  \left(  y+x\right)  ^{2}-1\right]
\exp \left \{  -\frac{\left(  y+x\right)  ^{2}}{4}\right \}  dy\\
\leq & \int_{1}^{+\infty}\left(  \ln y^{\ast}\right)  \exp \left \{
-\frac{\left(  y^{\ast}+x\right)  ^{2}}{4}\right \}  \cdot \left[  \left(
y+x\right)  ^{2}-1\right]  \exp \left \{  -\frac{\left(  y+x\right)  ^{2}}%
{4}\right \}  dy\\
= & \int_{1}^{+\infty}\frac{2}{y^{\ast}\left(  y^{\ast}+x\right)  }%
\exp \left \{  -\frac{\left(  y^{\ast}+x\right)  ^{2}}{4}\right \}  \cdot \left[
\left(  y+x\right)  ^{2}-1\right]  \exp \left \{  -\frac{\left(  y+x\right)
^{2}}{4}\right \}  dy\\
\leq & \frac{2}{1+x}\exp \left \{  -\frac{\left(  1+x\right)  ^{2}}{4}\right \}
\int_{1}^{+\infty}\left[  \left(  y+x\right)  ^{2}-1\right]  \exp \left \{
-\frac{\left(  y+x\right)  ^{2}}{4}\right \}  dy.
\end{array}
\]
Then by substitution $t=y+x$ and integration by parts, we have%
\[%
\begin{array}
[c]{rl}
& \frac{2}{1+x}\exp \left \{  -\frac{\left(  1+x\right)  ^{2}}{4}\right \}
\int_{1}^{+\infty}\left[  \left(  y+x\right)  ^{2}-1\right]  \exp \left \{
-\frac{\left(  y+x\right)  ^{2}}{4}\right \}  dy\\
= & \frac{2}{1+x}\exp \left \{  -\frac{\left(  1+x\right)  ^{2}}{4}\right \}
\left[  2\left(  1+x\right)  \exp \left \{  -\frac{\left(  1+x\right)  ^{2}}%
{4}\right \}  +\int_{1+x}^{+\infty}\exp \left \{  -\frac{t^{2}}{4}\right \}
dt\right] \\
\leq & \frac{2}{1+x}\exp \left \{  -\frac{\left(  1+x\right)  ^{2}}{4}\right \}
\left[  2\left(  1+x\right)  \exp \left \{  -\frac{\left(  1+x\right)  ^{2}}%
{4}\right \}  +\int_{1+x}^{+\infty}\frac{t}{1+x}\exp \left \{  -\frac{t^{2}}%
{4}\right \}  dt\right] \\
= & \frac{2}{1+x}\exp \left \{  -\frac{\left(  1+x\right)  ^{2}}{4}\right \}
\left[  2\left(  1+x\right)  \exp \left \{  -\frac{\left(  1+x\right)  ^{2}}%
{4}\right \}  +\frac{2}{1+x}\exp \left \{  -\frac{\left(  1+x\right)  ^{2}}%
{4}\right \}  \right] \\
= & \frac{4}{1+x}\left(  1+x+\frac{1}{1+x}\right)  \exp \left \{  -\frac{\left(
1+x\right)  ^{2}}{2}\right \}  ,
\end{array}
\]
which implies
\begin{equation}
\int_{1}^{+\infty}\left(  \ln y\right)  \left[  \left(  y+x\right)
^{2}-1\right]  \exp \left \{  -\frac{\left(  y+x\right)  ^{2}}{2}\right \}
dy\leq \frac{4}{1+x}\left(  1+x+\frac{1}{1+x}\right)  \exp \left \{
-\frac{\left(  1+x\right)  ^{2}}{2}\right \}  . \label{61202}%
\end{equation}
Thus, by (\ref{61201}) and (\ref{61202}), we obtain%
\[
\int_{0}^{+\infty}\left(  \ln y\right)  \left[  \left(  y+x\right)
^{2}-1\right]  \exp \left \{  -\frac{\left(  y+x\right)  ^{2}}{2}\right \}
dy\leq \left[  4+\frac{4}{\left(  x+1\right)  ^{2}}-\left(  x+1\right)
^{2}+1\right]  \exp \left \{  -\frac{\left(  1+x\right)  ^{2}}{2}\right \}  <0,
\]
which implies that $g_{x}^{\prime}\left(  0\right)  >0$. Therefore,
$g_{x}\left(  \lambda \right)  $ is convex and strictly increasing on $\left(
0,\frac{1}{2}\right)  $.

As for $f_{x}\left(  \lambda \right)  $, noting that%
\begin{equation}%
\begin{array}
[c]{rl}%
0< & \int_{0}^{+\infty}\left(  \ln \frac{y}{x-1}\right)  \left(  \ln \frac
{y}{x+1}\right)  y^{-2\lambda}\left[  \left(  y-x\right)  ^{2}-1\right]
\exp \left \{  -\frac{\left(  y-x\right)  ^{2}}{2}\right \}  dy\\
= & \frac{1}{4}f_{x}^{\prime \prime}\left(  \lambda \right)  +\frac{1}{2}%
\ln \left(  x^{2}-1\right)  f_{x}^{\prime}\left(  \lambda \right)  +\ln \left(
x-1\right)  \ln \left(  x+1\right)  f_{x}\left(  \lambda \right)  ,
\end{array}
\label{61204}%
\end{equation}
we consider the\ following two cases.

\textbf{Case 1:} If $\lim_{\lambda \rightarrow0}f_{x}^{\prime}\left(  \lambda \right)
>0$, we claim that $f_{x}^{\prime}\left(  \lambda \right)  >0$ for every
$\lambda \in \left(  0,\tilde{\lambda}_{2}\right)  $, where $\tilde{\lambda}%
_{2}\in \left(  0,\frac{1}{2}\right)  $ satisfies (\ref{62802}). Otherwise,
\[
\tilde{\lambda}:=\min \left \{  \lambda \in \left(  0,\tilde{\lambda}_{2}\right)
;\,f_{x}^{\prime}\left(  \lambda \right)  =0\right \}
\]
exists. Then by (\ref{61204}), (\ref{62802}) and the definition of $ \tilde{\lambda}$, we have $f_{x}%
^{\prime \prime}\left(  \tilde{\lambda}\right)  >0$, which contradicts the fact
that $f_{x}^{\prime}\left(  \lambda \right)  >0$ for $\lambda<\tilde{\lambda}$.
Consequently, $F_{x}^{\delta}\left(  \lambda \right)  =(1+\delta)f_{x}\left(
\lambda \right)  +g_{x}\left(  \lambda \right)  $ is increasing in $\lambda
\in \left(  0,\tilde{\lambda}_{2}\right)  $. Obviously, $F_{x}^{\delta}\left(
\lambda \right)  >0$ for $\lambda \in \left(  \tilde{\lambda}_{2},\frac{1}%
{2}\right)  $. Since $\lim_{\lambda \rightarrow0}F_{x}^{\delta}\left(
\lambda \right)  <0$, we conclude that there exists a unique $\lambda_{\delta
}\in \left(  0,\frac{1}{2}\right)  $ such that $F_{x}^{\delta}\left(
\lambda_{\delta}\right)  =0$.

\textbf{Case 2:} If $\lim_{\lambda \rightarrow0}f_{x}^{\prime}\left(  \lambda \right)
\leq0$, we have $\lim_{\lambda \rightarrow0}f_{x}^{\prime \prime}\left(
\lambda \right)  >0$ by (\ref{62803}) and (\ref{61204}). According to the proof
of Lemma \ref{61405}, we know that\
\[
\hat{\lambda}:=\min \left \{  \lambda \in \left(  0,\tilde{\lambda}_{2}\right)
;\text{ }f_{x}^{\prime}\left(  \lambda \right)  =0\right \}
\]
exists. By (\ref{61204}) we have $f_{x}^{\prime \prime}\left(  \lambda \right)
>0$ for $\lambda \in \left(  0,\hat{\lambda}\right]  $. For $\lambda \in \left(
\hat{\lambda},\tilde{\lambda}_{2}\right)  $, by similar analysis to Case 1, we
have $f_{x}^{\prime}\left(  \lambda \right)  >0$. Thus $F_{x}^{\delta}\left(
\lambda \right)  =(1+\delta)f_{x}\left(  \lambda \right)  +g_{x}\left(
\lambda \right)  $ is convex on $\left(  0,\hat{\lambda}\right)  $, increasing
in $\lambda \in \left(  \hat{\lambda},\tilde{\lambda}_{2}\right)  $, and
positive for $\lambda \in \left(  \tilde{\lambda}_{2},\frac{1}{2}\right)  $.
Consequently, there exists a unique $\lambda_{\delta}\in \left(  0,\frac{1}%
{2}\right)  $ such that $F_{x}^{\delta}\left(  \lambda_{\delta}\right)  =0$.

The proof of the uniqueness of $\tilde{\lambda}_{\delta}$ for $x\in \left[
2,+\infty \right)  $ is complete.

\bigskip

\textbf{(ii).} For $x\in \left(  1,2\right)  $, the uniqueness of
$\tilde{\lambda}_{\delta}$ can be proved by the same method as Step 2 in the
proof of Lemma \ref{61405}, the only difference being the proof of the
continuity of the following functions:%
\begin{equation}%
\begin{array}
[c]{c}%
\lambda_{1}^{\delta}\left(  x\right)  =\max \left \{  \lambda \in \left(
0,\frac{1}{2}\right)  :F_{x}^{\delta}\left(  \lambda \right)  =0\right \}
\text{,}\\
\lambda_{2}^{\delta}\left(  x\right)  =\min \left \{  \lambda \in \left(
0,\frac{1}{2}\right)  :F_{x}^{\delta}\left(  \lambda \right)  =0\right \}
\text{,}%
\end{array}
\label{61404}%
\end{equation}
which will be provided in Appendix, and then Step 2 is complete.

Hence we have proved that $z_{2}^{\lambda,1+\delta}$ is decreasing in
$\lambda \in \left(  0,\frac{1}{2}\right)  $. Letting $\delta \rightarrow0$, for
every $\lambda \in \left(  0,\frac{1}{2}\right)  $, we have $z^{\lambda}%
=\lim_{\delta \rightarrow0}z_{2}^{\lambda,1+\delta}$. Thus $z^{\lambda}$ is
decreasing in $\lambda \in \left(  0,\frac{1}{2}\right)  $. By Lemma \ref{2107},
we have $-x_{1}^{\lambda}<z^{\lambda}<x_{2}^{\lambda}$ for every $\lambda
\in \left(  0,\frac{1}{2}\right)  $. Thus, by Lemma \ref{2303}, Lemma
\ref{61405} and Lemma \ref{2503}, we have $\lim_{\lambda \rightarrow \frac{1}%
{2}}z^{\lambda}=1$.

\begin{appendix}
\section*{Appendix}
\setcounter{theorem}{0} \setcounter{equation}{0}
\renewcommand{\thetheorem}{A.\arabic{theorem}} \renewcommand{\theequation}{A.\arabic{equation}}
\begin{proposition}
\label{a1}Let $\lambda_{1}^{\delta}\left(  \cdot \right)  $ and $\lambda
_{2}^{\delta}\left(  \cdot \right)  $ be defined by (\ref{61404}). Then
$\lambda_{1}^{\delta}\left(  x\right)  $ and $\lambda_{2}^{\delta}\left(
x\right)  $ are continuous in $x\in \left(  1,2\right)  $.
\end{proposition}
Before giving the proof, we first state some lemmas which are useful in the
proof. Let $x\in \left(  1,2\right)  $ be any fixed constant. For $t\in \left[
-1,1\right]  $, define%
\begin{align}
l\left(  t\right)   &  :=\left(  \ln \frac{t+x}{x-1}\ln \frac{t+x}{x+1}\right)
\left(  t+x\right)  ^{-2\lambda},\label{f1}\\
h\left(  t\right)   &  :=\left(  t^{2}-1\right)  \exp \left \{  -\frac{t^{2}}%
{2}\right \}  +\left[  \left(  t+2x\right)  ^{2}-1\right]  \exp \left \{
-\frac{\left(  t+2x\right)  ^{2}}{2}\right \}  ,\label{f2}\\
g\left(  t\right)   &  :=\left(  \ln \frac{t+3x}{x-1}\ln \frac{t+3x}%
{x+1}\right)  \left(  t+3x\right)  ^{-2\lambda},\label{f3}\\
f\left(  t\right)   &  :=\left[  \left(  t+2x\right)  ^{2}-1\right]
\exp \left \{  -\frac{\left(  t+2x\right)  ^{2}}{2}\right \}  .\label{f4}%
\end{align}
\begin{lemma}
\label{25}Let $l\left(  t\right)  $ be defined by (\ref{f1}). Then there
exists a $c_{l}\in \left(  -1,0\right)  $ such that $l\left(  t\right)  $ is
decreasing in $t\in \left(  -1,c_{l}\right)  $ and increasing in $t\in \left(
c_{l},1\right)  $. Moreover, $l\left(  t\right)  <0$ for $t\in \left(
-1,1\right)  $.
\end{lemma}
\begin{proof}
It is obvious that $l\left(  t\right)  <0$ for $t\in \left(  -1,1\right)  $.

By simple calculation, we have%
\begin{equation}%
\begin{array}
[c]{rl}%
l^{\prime}\left(  t\right)  = & \left[  \ln \frac{\left(  t+x\right)  ^{2}%
}{x^{2}-1}-2\lambda \ln \frac{t+x}{x-1}\ln \frac{t+x}{x+1}\right]  \left(
t+x\right)  ^{-2\lambda-1}\\
= & \left[  L_{1}\left(  t\right)  +L_{2}\left(  t\right)  \right]  \left(
t+x\right)  ^{-2\lambda-1},
\end{array}
\label{61302}%
\end{equation}
where%
\[
L_{1}\left(  t\right)  :=\ln \frac{\left(  t+x\right)  ^{2}}{x^{2}-1}\text{ and
}L_{2}\left(  t\right)  :=-2\lambda \ln \frac{t+x}{x-1}\ln \frac{t+x}{x+1}.
\]

It is easy to check that $L_{1}\left(  \sqrt{x^{2}-1}-x\right)  =0$ and
$L_{1}^{\prime}\left(  t\right)  =2/\left(  t+x\right)  >0$ for $t\in \left(
-1,1\right)  $, which implies that $L_{1}\left(  t\right)  $ is increasing in
$t\in \left(  -1,1\right)  $, and
\[
\left \{
\begin{array}
[c]{ll}%
L_{1}\left(  t\right)  <0 & \text{ for }t\in \left(  -1,\sqrt{x^{2}%
-1}-x\right)  ,\\
L_{1}\left(  t\right)  >0 & \text{ for }t\in \left(  \sqrt{x^{2}-1}-x,1\right)
.
\end{array}
\right.
\]
For $L_{2}\left(  \cdot \right)  $, we have
\[
L_{2}^{\prime}\left(  t\right)  =-2\lambda \frac{1}{t+x}\ln \frac{\left(
t+x\right)  ^{2}}{x^{2}-1}=-2\lambda \frac{1}{t+x}L_{1}\left(  t\right)  .
\]
Thus $L_{2}^{\prime}\left(  \sqrt{x^{2}-1}-x\right)  =0$, and $L_{2}^{\prime
}\left(  t\right)  >0$ for $t\in \left(  -1,\sqrt{x^{2}-1}-x\right)  $,
$L_{2}^{\prime}\left(  t\right)  <0$ for $t\in \left(  \sqrt{x^{2}%
-1}-x,1\right)  $.

Therefore, $\left(  L_{1}+L_{2}\right)  \left(  t\right)  $
is increasing in $t\in \left(  -1,\sqrt{x^{2}-1}-x\right)  $. We can easily
check that $\left(  L_{1}+L_{2}\right)  \left(  -1\right)  <0$, and for
$t\in \left[  \sqrt{x^{2}-1}-x,1\right)  $,
\[
\left(  L_{1}+L_{2}\right)  \left(  t\right)  >\min_{t\in \left[  \sqrt
{x^{2}-1}-x,1\right)  }L_{1}\left(  t\right)  +\min_{t\in \left[  \sqrt
{x^{2}-1}-x,1\right)  }L_{2}\left(  t\right)  =0.
\]
Note that $\left(  t+x\right)  ^{-2\lambda-1}>0$ for $t\in \left(  -1,1\right)
$, then by (\ref{61302}), there exists a $c_{l}\in \left(  -1,\sqrt{x^{2}%
-1}-x\right)  \subseteq \left(  -1,0\right)  $ such that $l^{\prime}\left(
t\right)  <0$ for $t\in \left(  -1,c_{l}\right)  $ and $l^{\prime}\left(
t\right)  >0$ for $t\in \left(  c_{l},1\right)  $. \
\end{proof}
\begin{lemma}
\label{26}Let $h\left(  t\right)  $ be defined by (\ref{f2}). Then there
exists a $c_{h}\in \left(  0,1\right)  $ such that $h\left(  t\right)  $ is
decreasing in $t\in \left(  -1,c_{h}\right)  $ and increasing in $t\in \left(
c_{h},1\right)  $. Moreover, $h\left(  -1\right)  >0$, $h\left(  1\right)  >0$
and $h\left(  t\right)  <0$ for $t\in \left[  -0.65,0.7\right]  $.
\end{lemma}
\begin{proof}
Direct computation shows that%
\[%
\begin{array}
[c]{l}%
h^{\prime}\left(  t\right)  =\tilde{h}\left(  s\right)  +\tilde{h}\left(
t+2x\right)  ,\\
h^{\prime \prime}\left(  t\right)  =\left[  \left(  t^{2}-3\right)
^{2}-6\right]  \exp \left \{  -\frac{t^{2}}{2}\right \}  +\left \{  \left[
\left(  t+2x\right)  ^{2}-3\right]  ^{2}-6\right \}  \exp \left \{
-\frac{\left(  t+2x\right)  ^{2}}{2}\right \}  ,\\
h^{\prime \prime \prime}\left(  t\right)  =t\left[  10-\left(  t^{2}-5\right)
^{2}\right]  \exp \left \{  -\frac{t^{2}}{2}\right \}  +\left(  t+2x\right)
\left \{  10-\left[  \left(  t+2x\right)  ^{2}-5\right]  ^{2}\right \}
\exp \left \{  -\frac{\left(  t+2x\right)  ^{2}}{2}\right \}  .
\end{array}
\]
where $\tilde{h}\left(  s\right)  :=s\left(  3-s^{2}\right)  \exp \left \{
-\frac{s^{2}}{2}\right \}  $. We can easily check that%
\begin{equation}
\left \{
\begin{array}
[c]{l}%
\tilde{h}\left(  0\right)  =\tilde{h}\left(  \sqrt{3}\right)  =0,\\
\tilde{h}^{\prime}\left(  \pm \sqrt{3+\sqrt{6}}\right)  =\tilde{h}^{\prime
}\left(  \pm \sqrt{3-\sqrt{6}}\right)  =0,\\
\tilde{h}^{\prime}\left(  s\right)  >0\text{ \ for }s\in \left(  -\sqrt
{3-\sqrt{6}},\sqrt{3-\sqrt{6}}\right)  \cup \left(  \sqrt{3+\sqrt{6}},5\right)
,\\
\tilde{h}^{\prime}\left(  s\right)  <0\text{ \ for }s\in \left(  -1,-\sqrt
{3-\sqrt{6}}\right)  \cup \left(  \sqrt{3-\sqrt{6}},\sqrt{3+\sqrt{6}}\right)  .
\end{array}
\right.  \label{x}%
\end{equation}

\textbf{(i).} If $2x-\sqrt{3-\sqrt{6}}\geq \sqrt{3}$ and $x<2$, by (\ref{x}) we can
easily check that there exists a $c_{1}\in \left(  -1,1\right)  $ such that
$h^{\prime}\left(  t\right)  <0$ for $t\in \left(  -1,c_{1}\right)  $ and
$h^{\prime}\left(  t\right)  >0$ for $t\in \left(  c_{1},1\right)  $. Noting
that $h^{\prime}\left(  0\right)  <0$, we have $c_{1}>0$.

\textbf{(ii).} If $2x-\sqrt{3-\sqrt{6}}<\sqrt{3}$ and $x>1$, it is obvious that
$h^{\prime}\left(  t\right)  <0$ for $t\in \left(  -1,-\sqrt{3-\sqrt{6}%
}\right]  $. For $t\in \left(  -\sqrt{3-\sqrt{6}},\sqrt{3+\sqrt{6}}-2x\right]
$, we have $h^{\prime \prime \prime}\left(  t\right)  >0$, which implies that
$h^{\prime \prime}\left(  t\right)  $ is increasing on this interval. It is
easy to check that $h^{\prime \prime}\left(  -\sqrt{3-\sqrt{6}}\right)  <0$ and
$h^{\prime \prime}\left(  0\right)  >0$, so there exists a $c_{2}\in \left(
-\sqrt{3-\sqrt{6}},0\right]  $ such that $h^{\prime}\left(  t\right)  $ is
decreasing on $\left(  -\sqrt{3-\sqrt{6}},c_{2}\right)  $ and increasing on
$\left(  c_{2},\sqrt{3+\sqrt{6}}-2x\right)  $. Noting that $h^{\prime}\left(
-\sqrt{3-\sqrt{6}}\right)  <0$ and $h^{\prime}\left(  0\right)  <0$, which are
easy to check, we know that $h^{\prime}\left(  t\right)  <0$ for $t\in \left(
-\sqrt{3-\sqrt{6}},0\right]  $. From (\ref{x}) we have $h^{\prime \prime
}\left(  t\right)  >0$ for $t\in \left[  0,\sqrt{3+\sqrt{6}}-2x\right)  $, and
$h^{\prime}\left(  t\right)  $ is either increasing or decreasing on $\left(
\sqrt{3+\sqrt{6}}-2x,1\right)  $. Noting that $h^{\prime}\left(  0\right)  <0$
and $h^{\prime}\left(  1\right)  >0$, we can conclude that there exists a
$c_{3}\in \left(  0,1\right)  $ such that $h^{\prime}\left(  t\right)  <0$ for
$t\in \left(  -1,c_{3}\right)  $ and $h^{\prime}\left(  t\right)  >0$ for
$t\in \left(  c_{3},1\right)  $.

Moreover, it is easy to compute that $h\left(  -1\right)  >0$, $h\left(
1\right)  >0$, $h\left(  -0.65\right)  <0$ and $h\left(  0.7\right)  <0$,
which implies $h\left(  t\right)  <0$ for $t\in \left[  -0.65,0.7\right]  $.
\end{proof}
\begin{lemma}
\label{27}Let $g\left(  t\right)  $ be defined by (\ref{f3}). Then $g\left(
t\right)  $ is positive and increasing in $t\in \left(  -1,1\right)  $.
\end{lemma}
\begin{proof}
It is obvious that $g\left(  t\right)  >0$ for $t\in \left(  -1,1\right)  $.

We have
\[
g^{\prime}\left(  t\right)  =\left(  \ln \frac{t+3x}{x-1}+\ln \frac{t+3x}%
{x+1}-2\lambda \ln \frac{t+3x}{x-1}\ln \frac{t+3x}{x+1}\right)  \left(
t+3x\right)  ^{-2\lambda-1}.
\]
It is clear that, for $t\in \left(  -1,1\right)  $,
\[
0<\ln \frac{t+3x}{x+1}<1\text{ and }\ln \frac{t+3x}{x-1}=\ln \frac{t+3x}{x+1}%
+\ln \frac{x+1}{x-1}>0\text{.}%
\]
Thus%
\[%
\begin{array}
[c]{rl}%
g^{\prime}\left(  t\right)  = & \left \{  2\ln \frac{t+3x}{x+1}+\ln \frac
{x+1}{x-1}-2\lambda \left[  \left(  \ln \frac{t+3x}{x+1}\right)  ^{2}+\ln
\frac{t+3x}{x+1}\ln \frac{x+1}{x-1}\right]  \right \}  \left(  t+3x\right)
^{-2\lambda-1}\\
> & \left \{  2\ln \frac{t+3x}{x+1}+\ln \frac{x+1}{x-1}-\left[  \left(  \ln
\frac{t+3x}{x+1}\right)  ^{2}+\ln \frac{t+3x}{x+1}\ln \frac{x+1}{x-1}\right]
\right \}  \left(  t+3x\right)  ^{-2\lambda-1}>0.
\end{array}
\]
\end{proof}
\begin{lemma}
\label{28}Let $x\in \left(  1,2\right)  $ and $\delta \in \left(  0,+\infty
\right)  $ be any fixed constants, and let $F_{x}^{\delta}\left(
\lambda \right)  $ be defined by (\ref{1902}) for $\lambda \in \left(  0,\frac
{1}{2}\right)  $. If there exists a $\mu_{\delta}\in \left(  0,\frac{1}%
{2}\right)  $ such that $F_{x}^{\delta}\left(  \mu_{\delta}\right)  =0$ and
$\left(  F_{x}^{\delta}\right)  ^{\prime}\left(  \mu_{\delta}\right)  =0$,
then $\left(  F_{x}^{\delta}\right)  ^{\prime \prime}\left(  \mu_{\delta
}\right)  >0$.
\end{lemma}
\begin{proof}
For $\lambda \in \left(  0,\frac{1}{2}\right)  $, set
\[%
\begin{array}
[c]{cl}%
\Gamma^{\delta}\left(  \lambda,x\right)  := & \left(  1+\delta \right)
\int_{0}^{+\infty}\left(  \ln \frac{y}{x-1}\ln \frac{y}{x+1}\right)
y^{-2\lambda}\left[  \left(  y-x\right)  ^{2}-1\right]  \exp \left \{
-\frac{\left(  y-x\right)  ^{2}}{2}\right \}  dy\\
& +\int_{0}^{+\infty}\left(  \ln \frac{y}{x-1}\ln \frac{y}{x+1}\right)
y^{-2\lambda}\left[  \left(  y+x\right)  ^{2}-1\right]  \exp \left \{
-\frac{\left(  y+x\right)  ^{2}}{2}\right \}  dy.
\end{array}
\]
Noting that%
\[
\Gamma^{\delta}\left(  \lambda,x\right)  =\frac{1}{4}\left(  F_{x}^{\delta
}\right)  ^{\prime \prime}\left(  \lambda \right)  +\frac{1}{2}\ln \left(
x^{2}-1\right)  \left(  F_{x}^{\delta}\right)  ^{\prime}\left(  \lambda
\right)  +\ln \left(  x-1\right)  \ln \left(  x+1\right)  F_{x}^{\delta}\left(
\lambda \right)  ,
\]
we only need to prove that $\Gamma^{\delta}\left(  \lambda,x\right)  >0$ for
every $\lambda \in \left(  0,\frac{1}{2}\right)  $ and \thinspace$x\in \left(
1,2\right)  $. Since
\[
\int_{0}^{+\infty}\left(  \ln \frac{y}{x-1}\ln \frac{y}{x+1}\right)
y^{-2\lambda}\left[  \left(  y-x\right)  ^{2}-1\right]  \exp \left \{
-\frac{\left(  y-x\right)  ^{2}}{2}\right \}  dy>0
\]
which is obvious, it suffices to prove that%
\[%
\begin{array}
[c]{cl}%
\Gamma \left(  \lambda,x\right)  := & \int_{0}^{+\infty}\left(  \ln \frac
{y}{x-1}\ln \frac{y}{x+1}\right)  y^{-2\lambda}\left[  \left(  y-x\right)
^{2}-1\right]  \exp \left \{  -\frac{\left(  y-x\right)  ^{2}}{2}\right \}  dy\\
& +\int_{0}^{+\infty}\left(  \ln \frac{y}{x-1}\ln \frac{y}{x+1}\right)
y^{-2\lambda}\left[  \left(  y+x\right)  ^{2}-1\right]  \exp \left \{
-\frac{\left(  y+x\right)  ^{2}}{2}\right \}  dy
\end{array}
\]
is positive for every $\lambda \in \left(  0,\frac{1}{2}\right)  $ and
\thinspace$x\in \left(  1,2\right)  $.

\textbf{(i).} Let $x\in \left(  1,1.5\right)  $ be any fixed constant. Now we show that
$\Gamma \left(  \lambda,x\right)  >0$ for every $\lambda \in \left(  0,\frac
{1}{2}\right)  $.

It is easy to check that%
\begin{equation}%
\begin{array}
[c]{rl}%
\Gamma \left(  \lambda,x\right)  > & \int_{x-1}^{x+1}\left(  \ln \frac{y}%
{x-1}\ln \frac{y}{x+1}\right)  y^{-2\lambda}\left[  \left(  y+x\right)
^{2}-1\right]  \exp \left \{  -\frac{\left(  y+x\right)  ^{2}}{2}\right \}  dy\\
& +\int_{x-1}^{x+1}\left(  \ln \frac{y}{x-1}\ln \frac{y}{x+1}\right)
y^{-2\lambda}\left[  \left(  y-x\right)  ^{2}-1\right]  \exp \left \{
-\frac{\left(  y-x\right)  ^{2}}{2}\right \}  dy\\
& +\int_{3x-1}^{3x+1}\left(  \ln \frac{y}{x-1}\ln \frac{y}{x+1}\right)
y^{-2\lambda}\left[  \left(  y-x\right)  ^{2}-1\right]  \exp \left \{
-\frac{\left(  y-x\right)  ^{2}}{2}\right \}  dy\\
= & \int_{-1}^{1}\left[  l\left(  t\right)  h\left(  t\right)  +g\left(
t\right)  f\left(  t\right)  \right]  dt,
\end{array}
\label{1701}%
\end{equation}
where we use integration by substitution in the last equation and $l\left(
t\right)  $, $h\left(  t\right)  $, $g\left(  t\right)  $, $f\left(  t\right)
$ are defined by (\ref{f1})-(\ref{f4}). So it suffices to prove
\[
\int_{-1}^{1}\left[  l\left(  t\right)  h\left(  t\right)  +g\left(  t\right)
f\left(  t\right)  \right]  dt>0.
\]

Noting that
\begin{equation}%
\begin{array}
[c]{ll}
& \int_{-1}^{1}\left[  l\left(  t\right)  h\left(  t\right)  +g\left(
t\right)  f\left(  t\right)  \right]  dt\\
= & \int_{-1}^{0.4}\left[  l\left(  t\right)  h\left(  t\right)  +g\left(
t\right)  f\left(  t\right)  \right]  dt+\int_{0.4}^{0.7}\left[  l\left(
t\right)  h\left(  t\right)  +g\left(  t\right)  f\left(  t\right)  \right]
dt\\
& +\int_{0.7}^{1}\left[  l\left(  t\right)  h\left(  t\right)  +g\left(
t\right)  f\left(  t\right)  \right]  dt,
\end{array}
\label{61303}%
\end{equation}
we consider the integrals respectively.
According to Lemma \ref{25}-Lemma \ref{27} and the fact that $f\left(  t\right)  >0$
for $t\in \left(  -1,1\right)  $, we have
\begin{equation}
\int_{0.4}^{0.7}\left[  l\left(  t\right)  h\left(  t\right)  +g\left(
t\right)  f\left(  t\right)  \right]  dt>0.
\end{equation}
For the last term of (\ref{61303}), we note that%
\[%
\begin{array}
[c]{c}%
l\left(  0.7\right)  +g\left(  0.7\right)  >\ln \frac{x+0.7}{x-1}\ln
\frac{x+0.7}{x+1}+\left(  \ln \frac{3x+0.7}{x-1}\ln \frac{3x+0.7}{x+1}\right)
\frac{1}{3x+0.7}>0
\end{array}
.
\]
Thus $g\left(  0.7\right)  >-l\left(  0.7\right)  >0$, that is, $g\left(
t\right)  /l\left(  t\right)  <-1$. It is clear that $h\left(  t\right)
<f\left(  t\right)  $ for $t\in \left(  0.7,1\right)  $. Since $g\left(
t\right)  $ and $l\left(  t\right)  $ are increasing in $t\in \left(
0.7,1\right)  $ and $l\left(  t\right)  <0$ for $t\in \left(  0.7,1\right)  $
according to Lemma \ref{25} and Lemma \ref{27}, we have%
\begin{equation}
\int_{0.7}^{1}\left[  l\left(  t\right)  h\left(  t\right)  +g\left(
t\right)  f\left(  t\right)  \right]  dt=\int_{0.7}^{1}l\left(  t\right)
\left[  h\left(  t\right)  +\frac{g\left(  t\right)  }{l\left(  t\right)
}f\left(  t\right)  \right]  dt>\int_{0.7}^{1}l\left(  t\right)  \left[
h\left(  t\right)  -f\left(  t\right)  \right]  dt>0. \label{61304}%
\end{equation}
For the first term of (\ref{61303}), since $g\left(  t\right)  >0$ and
$f\left(  t\right)  >0$ for $t\in \left(  -1,1\right)  $, we can check that%
\begin{equation}%
\begin{array}
[c]{ll}
& \int_{-1}^{0.4}\left[  l\left(  t\right)  h\left(  t\right)  +g\left(
t\right)  f\left(  t\right)  \right]  dt\\
\geq & \int_{-1}^{-0.65}l\left(  t\right)  h\left(  t\right)  dt+\int
_{-0.65}^{-0.4}l\left(  t\right)  h\left(  t\right)  dt+\int_{-0.4}%
^{1-x}l\left(  t\right)  h\left(  t\right)  dt+\int_{1-x}^{0}l\left(
t\right)  h\left(  t\right)  dt+\int_{0}^{0.4}l\left(  t\right)  h\left(
t\right)  dt\\
> & \left[  \left(  2x-1\right)  ^{2}-1\right]  \exp \left \{  -\frac{\left(
2x-1\right)  ^{2}}{2}\right \}  \int_{-1}^{-0.65}\ln \frac{x+t}{x-1}\ln
\frac{x+t}{x+1}\left(  x+t\right)  ^{-1}dt\\
& +\left \{  \left[  \left(  2x-0.4\right)  ^{2}-1\right]  \exp \left \{
-\frac{\left(  2x-0.4\right)  ^{2}}{2}\right \}  +\left(  0.4^{2}-1\right)
\exp \left \{  -\frac{0.4^{2}}{2}\right \}  \right \}  \int_{-0.4}^{1-x}\ln
\frac{x+t}{x-1}\ln \frac{x+t}{x+1}dt\\
& +\left \{  \left[  \left(  x+1\right)  ^{2}-1\right]  \exp \left \{
-\frac{\left(  x+1\right)  ^{2}}{2}\right \}  +\left[  \left(  x-1\right)
^{2}-1\right]  \exp \left \{  -\frac{\left(  x-1\right)  ^{2}}{2}\right \}
\right \}  \int_{1-x}^{0}\ln \frac{x+t}{x-1}\ln \frac{x+t}{x+1}\left(
x+t\right)  ^{-1}dt\\
& +\left \{  \left[  \left(  2x+0.4\right)  ^{2}-1\right]  \exp \left \{
-\frac{\left(  2x+0.4\right)  ^{2}}{2}\right \}  +\left(  0.4^{2}-1\right)
\exp \left \{  -\frac{0.4^{2}}{2}\right \}  \right \}  \int_{0}^{0.4}\ln \frac
{x+t}{x-1}\ln \frac{x+t}{x+1}\left(  x+t\right)  ^{-1}dt\\
> & 0.
\end{array}
\label{1702}%
\end{equation}
Then, by (\ref{61303})-(\ref{1702}), we conclude that $\int_{-1}^{1}\left[
l\left(  t\right)  h\left(  t\right)  +g\left(  t\right)  f\left(  t\right)
\right]  dt>0$, which implies $\Gamma \left(  \lambda,x\right)  >0$ for every
$\lambda \in \left(  0,\frac{1}{2}\right)  $ and \thinspace$x\in \left(
1,1.5\right)  $ by (\ref{1701}).

\textbf{(ii).} Let $x\in \left[  1.5,2\right)  $ be any fixed constant. We next prove
$\Gamma \left(  \lambda,x\right)  >0$ for every $\lambda \in \left(  0,\frac
{1}{2}\right)  $.

It is clear that
\[
\Gamma \left(  \lambda,x\right)  >\int_{-1.5}^{-1}\left(  \ln \frac{t+x}{x-1}%
\ln \frac{t+x}{x+1}\right)  \left(  t+x\right)  ^{-2\lambda}\left(
t^{2}-1\right)  \exp \left \{  -\frac{t^{2}}{2}\right \}  dt+\int_{-1}^{1}\left[
l\left(  t\right)  h\left(  t\right)  +g\left(  t\right)  f\left(  t\right)
\right]  dt\text{.}%
\]
Noting that (\ref{61304}) also holds for $x\in \left[  1.5,2\right)  $, and
$\int_{-0.5}^{0.7}\left[  l\left(  t\right)  h\left(  t\right)  +g\left(
t\right)  f\left(  t\right)  \right]  dt>0$ according to Lemma \ref{25}%
-Lemma \ref{27}, we have%
\[%
\begin{array}
[c]{rl}%
\Gamma \left(  \lambda,x\right)  > & \int_{-1.5}^{-1}\left(  \ln \frac{t+x}%
{x-1}\ln \frac{t+x}{x+1}\right)  \left(  t+x\right)  ^{-2\lambda}\left(
t^{2}-1\right)  \exp \left \{  -\frac{t^{2}}{2}\right \}  dt+\int_{-1}%
^{-0.5}\left[  l\left(  t\right)  h\left(  t\right)  +g\left(  t\right)
f\left(  t\right)  \right]  dt\\
> & \int_{x-1.5}^{x-1}\ln \frac{y}{x-1}\ln \frac{y}{x+1}y^{-2\lambda}\left[
\left(  y+x\right)  ^{2}-1\right]  \exp \left \{  -\frac{\left(  y+x\right)
^{2}}{2}\right \}  dy\\
& +\int_{x-1}^{x-0.5}\ln \frac{y}{x-1}\ln \frac{y}{x+1}y^{-2\lambda}\left[
\left(  y+x\right)  ^{2}-1\right]  \exp \left \{  -\frac{\left(  y+x\right)
^{2}}{2}\right \}  dy\\
= & \int_{0}^{0.5}\left(  \ln \frac{-t+x-1}{x-1}\ln \frac{-t+x-1}{x+1}\right)
\left(  -t+x-1\right)  ^{-2\lambda}\left[  \left(  -t+2x-1\right)
^{2}-1\right]  \exp \left \{  -\frac{\left(  -t+2x-1\right)  ^{2}}{2}\right \}
dt\\
& +\int_{0}^{0.5}\left(  \ln \frac{t+x-1}{x-1}\ln \frac{t+x-1}{x+1}\right)
\left(  t+x-1\right)  ^{-2\lambda}\left[  \left(  t+2x-1\right)
^{2}-1\right]  \exp \left \{  -\frac{\left(  -t+2x-1\right)  ^{2}}{2}\right \}
dt,
\end{array}
\]
by integration by substitution. For $t\in \left(  0,0.5\right)  $, we can
easily check that%
\[
-\ln \frac{-t+x-1}{x-1}>\ln \frac{t+x-1}{x-1}>0,
\]%
\[
\ln \frac{-t+x-1}{x+1}<\ln \frac{t+x-1}{x+1}<0,
\]%
\[
\left(  -t+x-1\right)  ^{-2\lambda}>\left(  t+x-1\right)  ^{-2\lambda},
\]
and%
\[
\left[  \left(  -t+2x-1\right)  ^{2}-1\right]  \exp \left \{  -\frac{\left(
-t+2x-1\right)  ^{2}}{2}\right \}  >\left[  \left(  t+2x-1\right)
^{2}-1\right]  \exp \left \{  -\frac{\left(  -t+2x-1\right)  ^{2}}{2}\right \}
.
\]
Then we have
\[%
\begin{array}
[c]{l}%
\int_{0}^{0.5}\left(  \ln \frac{-t+x-1}{x-1}\ln \frac{-t+x-1}{x+1}\right)
\left(  -t+x-1\right)  ^{-2\lambda}\left[  \left(  -t+2x-1\right)
^{2}-1\right]  \exp \left \{  -\frac{\left(  -t+2x-1\right)  ^{2}}{2}\right \}
dt\\
+\int_{0}^{0.5}\left(  \ln \frac{t+x-1}{x-1}\ln \frac{t+x-1}{x+1}\right)
\left(  t+x-1\right)  ^{-2\lambda}\left[  \left(  t+2x-1\right)
^{2}-1\right]  \exp \left \{  -\frac{\left(  -t+2x-1\right)  ^{2}}{2}\right \}
dt>0,
\end{array}
\]
which implies $\Gamma \left(  \lambda,x\right)  >0$.
\end{proof}

\begin{proof}
[\textnormal{\textbf{Proof of Proposition \ref{a1}}}]Let $x \in\left( 1,2 \right)$ be any fixed constant. By Lemma \ref{28}, for every
$\varepsilon>0$, there exists a $\lambda \in \left(  \lambda_{2}^{\delta}\left(
x\right)  ,\lambda_{2}^{\delta}\left(  x\right)  +\varepsilon \right)  $ such
that $F_{x}^{\delta}\left(  \lambda \right)  >0$. Then by Lemma \ref{2107}, there
exists $x^{\prime}\in \left(  1,x\right)  $ such that $F_{x^{\prime}}^{\delta
}\left(  \lambda \right)  =0$. It follows that $\lambda_{2}^{\delta}\left(
x^{\prime}\right)  \leq \lambda$ by the definition of $\lambda_{2}^{\delta
}\left( \cdot\right)  $. Since $\lambda_{2}^{\delta}\left(  x\right)  $ is
decreasing in $x\in \left(  1,+\infty \right)  $, which can be proved by the
same method as (ii) in the proof of Lemma \ref{61405}, we have $\lambda
_{2}^{\delta}\left(  x\right)  <\lambda_{2}^{\delta}\left(  x^{\prime}\right)
\leq \lambda$, which implies the left-continuity of $\lambda_{2}^{\delta
}\left(  x\right)  $ in $x\in \left(  1,2\right)  $. Thus we conclude that
$\lambda_{2}^{\delta}\left(  x\right)  $ is continuous in $x\in \left(
1,2\right)  $ since $\lambda_{2}^{\delta}\left(  x\right)  $ is obviously
right-continuous. The proof for the continuity of $\lambda_{1}^{\delta}\left(
x\right)  $ in $x\in \left(  1,2\right)  $ is similar.
\end{proof}
\end{appendix}

\end{document}